\documentclass[12pt,reqno]{amsart}
\usepackage{bm,cite,amssymb,fullpage,mathabx,color,enumerate}
\usepackage{tikz,tikz-3dplot}
\usepackage[colorlinks=true,linkcolor=blue,citecolor=blue]{hyperref}

\newtheorem{lemma}{Lemma}
\newtheorem{theorem}{Theorem}
\newtheorem{prop}{Proposition}
\newtheorem*{cor}{Corollary}

\theoremstyle{remark}
\newtheorem{remark}{Remark}[section]

\title{Discrete maximal operators over surfaces of higher codimension}
\author[T.C. Anderson]{Theresa C. Anderson}
\address{Department of Mathematics, Purdue University, 150 N. University Street, West Lafayette, IN 47906}
\email{tcanderson@purdue.edu}
\author[A.V. Kumchev]{Angel V. Kumchev}
\address{Department of Mathematics, Towson University, 8000 York Road, Towson, MD 21252}
\email{akumchev@towson.edu}
\author[E.A. Palsson]{Eyvindur A. Palsson}
\address{Department of Mathematics, Virginia Tech, 225 Stanger Street, Blacksburg, VA 24061}
\email{palsson@vt.edu}

\date{\today}

\numberwithin{equation}{section}

\newcommand{\dalpha}{\mathrm d\bm\alpha}
\newcommand{\dbeta}{\mathrm d\bm\beta}

\newcommand{\eps}{\varepsilon}
\DeclareMathOperator{\lcm}{lcm}

\begin{document}

\begin{abstract}
Integration over curved manifolds with higher codimension and, separately, discrete variants of continuous operators, have been two important, yet separate themes in harmonic analysis, discrete geometry and analytic number theory research. Here we unite these themes to study discrete analogues of operators involving higher (intermediate) codimensional integration.  We consider a maximal operator that averages over triangular configurations and prove several bounds that are close to optimal. A distinct feature of our approach is the use of multilinearity to obtain nontrivial $\ell^1$-estimates by a rather general idea that is likely to be applicable to other problems. 
\end{abstract}

\maketitle

\section{Introduction}

Operators involving integration along a curved smooth manifold have been a central theme in harmonic analysis and related fields.  Curvature adds subtlety to the analysis of such operators; for example, celebrated bounds for the spherical maximal function by Stein \cite{Stein76} and Bourgain \cite{Bour85} are significantly more delicate than the respective bounds for the classical Hardy--Littlewood maximal operator on Euclidean space. Operators involving integration over a curved manifold of codimension 1 or $d-1$ in $\mathbb R^d$ have been extensively studied in a variety of contexts, already providing a wide range of challenges; see for example \cite{TaoWright} and the references therein. When the integration involves a manifold of intermediate codimension, the analysis becomes even more intricate and involved and the problem of bounding such operators turns into a much more difficult problem. It is therefore not surprising that results for such operators are more scarce in the literature.

Another area of extensive study involves discrete variants of continuous operators. Initiated by work of Bourgain \cite{Bourgain_ergodic} in ergodic theory, research in this direction has continued to evolve into a standalone subfield of harmonic analysis following the pivotal work of Magyar, Stein and Wainger \cite{MSW}, where they considered the discrete analogue of the spherical maximal function. Several authors have proved maximal and/or improving inequalities for discrete operators over lattice points on surfaces of arithmetic interest; see \cite{ACHK2, ACHK, BMSW, Cook19, HeHu19, Hugh17a, Hugh17b, Ione04, KeLa20, Magy97, Magy02, Pierce_mean_values} for some such results. A distinctive feature of such work is the interplay between analysis and number theory, as the arithmetic properties of the underlying discrete set play a central role when the analogous continuous operator involves curvature. Indeed, in almost all cases, even the asymptotics for the size of the underlying set of lattice points lead to number-theoretic problems with a long and rich history.    

In this paper, we consider a problem that belongs to both of these bodies of research. We study a discrete averaging operator, where we average over equilateral triangles with vertices in $\mathbb Z^d$: namely,
\begin{equation}\label{eqi.1}
(\#\mathcal V_\lambda)^{-1} \sum_{(\mathbf{u,v}) \in \mathcal V_\lambda} f(\mathbf x-\mathbf u)g(\mathbf x-\mathbf v),
\end{equation}
where the summation is over the point set
\begin{align*}
\mathcal V_\lambda &= \big\{ (\mathbf{u, v}) \in \mathbb Z^d \times \mathbb Z^d : |\mathbf u|^2 = |\mathbf v|^2 = |\mathbf u - \mathbf v|^2 = \lambda \big\} \\
&= \big\{ (\mathbf{u, v}) \in \mathbb Z^d \times \mathbb Z^d : |\mathbf u|^2 = |\mathbf v|^2 = 2\mathbf u\cdot \mathbf v = \lambda \big\}, 
\end{align*} 
$| \cdot |$ being the Euclidean norm on $\mathbb R^d$. It is clear from the second representation of $\mathcal V_\lambda$ that $\mathcal V_\lambda = \varnothing$ for odd $\lambda$. On the other hand, when $\lambda$ is a large even integer and the dimension $d$ is not too small, one expects that $\#\mathcal V_\lambda \asymp \lambda^{d-3}$. This bound certainly holds in the dimensions we consider, for example from the results of Raghavan \cite{Ragh59} (or from Theorem~\ref{thm2} below). Thus, we may replace the operator \eqref{eqi.1} with 
\begin{equation}\label{eqi.2}
T_\lambda(f,g)(\mathbf x) = \lambda^{3-d} \sum_{(\mathbf{u,v}) \in \mathcal V_\lambda} f(\mathbf x-\mathbf u)g(\mathbf x-\mathbf v),
\end{equation}
which is slightly more convenient to work with. 

The motivation for studying this particular operator comes from point configuration questions that generalize the Erd\H{o}s distance problem and its continuous analogue, the Falconer distance problem. In the continuous setting, specific bounds on such averaging operators have been used to establish Falconer type theorems for triangles \cite{GI12,GGIP15}, as well as having been studied independently \cite{PS20}. In the setting of $\mathbb{Z}^d$, a precursor to the operator we study appeared in the work of Magyar \cite{Magy09}, where he established a Ramsey type theorem for simplices by building on his earlier work for distances \cite{Magy08}.  We also mention that we have recently learned of forthcoming related independent work \cite{CLM}.

Our main results---Theorem \ref{thm1} below and its corollary---establish that the corresponding maximal operator is bounded from $\ell^p(\mathbb Z^d) \times \ell^q(\mathbb Z^d)$ to $\ell^r(\mathbb Z^d)$ for a range of choices for $p,q,r$. To the best of our knowledge, these are the first examples of discrete maximal inequalities where the underlying continuous manifold has codimension greater than 1. In analogy with the classical theory of interpolation of operators between $L^p$ spaces, we say that a bounded operator $T$ that maps $\ell^{p}(\mathbb{Z}^d)\times\ell^{q}(\mathbb{Z}^d)$ into $\ell^{r}(\mathbb{Z}^d)$ is of \emph{type $(p,q; r)$}. In this terminology, we prove the following.

\begin{theorem}\label{thm1}
Let $d\geq 9$ and $p > p_0(d) = \max\big( \frac {32}{d+8}, \frac {d+4}{d-2} \big)$. Then the maximal operator
\[ T^*(f,g) = \sup_{\lambda \in \mathbb N} |T_\lambda(f,g)| \] 
is of type $(p, \infty; p)$.
\end{theorem}

By symmetry, $T^*$ is of course also of type $(\infty, p; p)$. Interpolation between these two results shows that $T^*$ is of type $(p, q; r)$ whenever $r > p_0(d)$ and $\frac 1p + \frac 1q = \frac 1r$. Recalling that $\ell^p(\mathbb Z^d)$ spaces increase with $p$, we obtain the following corollary on the boundedness of~$T^*$.

\begin{cor}
Let $d \ge 9$ and $p_0(d)$ be as above. The maximal operator $T^*$ is of type $(p, q; r)$ whenever $r > p_0(d)$ and $1 \le p,q \le \infty$ with $\frac 1p + \frac 1q \ge \frac 1r$.
\end{cor}

The full range of triples $(p,q,r)$ for which this corollary establishes the boundedness of $T^*$ is depicted on Figure \ref{fig1}. Each triple $(p,q,r)$ is represented by the point $(\frac 1p, \frac 1q, \frac 1r)$ in the unit cube. The corollary applies to all the triples $(p,q,r)$ for which the respective point lies in the displayed solid polyhedron, with exception of its top face (colored red).  

\begin{figure}\label{fig1}
\tdplotsetmaincoords{70}{110}
\begin{tikzpicture}[tdplot_main_coords]
\draw[very thick,dotted,->] (0,0,0) -- (6,0,0) node[anchor=east]{$\frac 1p$};
\draw[very thick,dotted,->] (0,0,0) -- (0,5.5,0) node[anchor=west]{$\frac 1q$};
\draw[very thick,dotted,->] (0,0,0) -- (0,0,4) node[anchor=west]{$\frac 1r$};
\fill[fill=gray!20!white, fill opacity=0.7] (0,5,3) -- (0,3,3) -- (0,0,0) -- (0,5,0) -- (0,5,3);
\fill[fill=gray!20!white, fill opacity=0.7] (5,0,3) -- (3,0,3) -- (0,0,0) -- (5,0,0) -- (5,0,3);
\draw[thick, dashed, fill=gray!20!white, fill opacity=0.7] (0,0,0) -- (0,3,3) -- (3,0,3) -- (0,0,0); 
\fill[fill=gray!20!white, fill opacity=0.9] (5,5,0) -- (0,5,0) -- (0,0,0) -- (5,0,0) -- (5,5,0);
\draw[thick, dashed] (5,0,0) -- (0,0,0) -- (0,5,0);
\fill[thin, fill=gray!20!white, fill opacity=0.7] (5,0,3) -- (5,0,0) -- (5,5,0) -- (0,5,0) -- (0,5,3) -- (5,5,3) -- (5,0,3);
\draw[thick, red, fill=red!10!white, fill opacity=0.7] (0,5,3) -- (0,3,3) -- (3,0,3) -- (5,0,3) -- (5,5,3) -- (0,5,3);
\draw[thick] (5,0,3) -- (5,0,0) -- (5,5,0) -- (0,5,0) -- (0,5,3) node[anchor=west] {$\big(0,1,\frac 1{p_0}\big)$};
\draw[thick] (5,5,3) -- (5,5,0);
\draw (0,3,3) node[anchor=south] {$\big(0,\frac 1{p_0},\frac 1{p_0}\big)$};
\draw (3,0,3) node[anchor=south east] {$\big(\frac 1{p_0},0,\frac 1{p_0}\big)$};
\end{tikzpicture}
\caption{Points $(\frac{1}{p},\frac{1}{q}, \frac 1r)$ with $T^{*}$ of type $(p, q; r)$.} 
\end{figure}
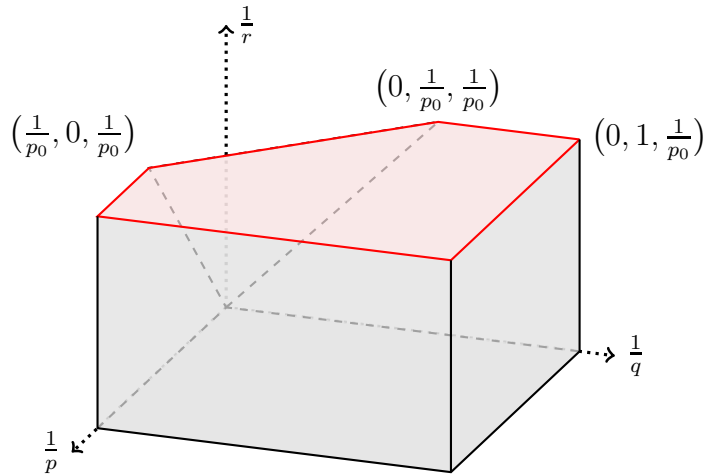

It is natural to ask how close these results are to being best possible. The condition $p > d/(d-3)$ appears at several places in our argument in ways that suggest that it may be a natural barrier for the problem. We have more to say about this, but we defer such discussion to the last section of the paper. If one accepts this restriction, however, and also insists that the range of $p$ include $p=2$, then the condition $d \ge 7$ on the dimension quickly emerges. 

We should point out that to reach the full strength of the results stated above we rely substantially on the multilinearity of the operator $T_\lambda$. In particular, unlike much of the existing work on discrete maximal operators, we leverage this multilinearity to obtain non-trivial $\ell^1$-bounds, which we combine with more traditional $\ell^2$-bounds. Without this idea, we would have to increase the value of $p_0(d)$ to $(d+16)/(d+4)$.  

\medskip

As in past work on discrete averages over surfaces of codimension 1, bounds on $\ell^2(\mathbb Z^d)$ play a central role in our arguments. To that end, we analyze the Fourier multiplier of $T_\lambda$, 
\[ \widehat{T_\lambda}(\bm\xi, \bm\eta) = \lambda^{3-d}\sum_{(\mathbf{u,v}) \in \mathcal V_\lambda} e(\bm\xi \cdot \mathbf u + \bm\eta \cdot \mathbf v), \] 
where $e(x) = e^{2\pi ix}$. Raghavan \cite{Ragh59} used the theory of Siegel modular forms to prove general results on simultaneous representations of integers by positive definite quadratic forms. Raghavan's work yields an asymptotic formula for $\widehat{T_\lambda} (\bm0, \bm0)$ when $d \ge 7$. His results were later improved on by Kitaoka in a series of papers during the 1980s. In particular, Kitaoka \cite{Kita88} showed that the asymptotic formula for $\widehat{T_\lambda} (\bm0, \bm0)$ holds when $d = 6$. The reader interested in this topic should see also important work of Hsia, Kitaoka, and Kneser \cite{HKK} and Ellenberg and Venkatesh \cite{EV} that uses $p$-adic methods. More recently, Dietmann and Harvey \cite{DiHa13} and Brandes \cite{Bran15} applied a version of the circle method pioneered by Davenport \cite{Dave59} and Birch \cite{Birc61} to generalize Raghavan's theorem to forms of arbitrary degree $k \ge 2$; their work gives an asymptotic formula for $\widehat{T_\lambda} (\bm0, \bm0)$ when $d \ge 13$ (see \cite[Theorem 1.1]{Bran15}). In this paper, we apply the Hardy--Littlewood circle method directly to the Diophantine equations defining $\mathcal V_\lambda$. This allows us to make use of moment estimates for exponential sums and to extend Raghavan's asymptotic to the general multiplier $\widehat{T_\lambda}(\bm\xi, \bm\eta)$ for all $d \ge 7$. More importantly, when $d \ge 9$, we are able to leverage our approximation for the multiplier to an approximation for the operator in $\ell^p(\mathbb Z^d)$, $p > p_0(d)$. 

In order to state our approximation results, we need to introduce some notation. For vectors $\mathbf{x, y} \in \mathbb R^k$, we write
\[ s(\mathbf x) = x_1 + x_2 + \dots + x_k, \quad  \bm\phi(\mathbf{ x,y}) = \big( |\mathbf x|^2, 2\mathbf x \cdot \mathbf y, |\mathbf y|^2  \big). \]
When $q, m, n \in \mathbb N$, $\mathbf a \in \mathbb Z^3$, $\bm\alpha \in \mathbb T^3$, $\xi,\eta \in \mathbb T$, we define
\begin{gather}
g(q; \mathbf a, m, n) = q^{-2} \sum_{r=1}^q \sum_{s=1}^q e_q( \mathbf a \cdot \bm\phi(r,s) + mr + ns), \label{eq1.3}\\
V_N(\bm\alpha; \xi, \eta) = \int_{-N}^N \int_{-N}^N e(\bm\alpha \cdot \bm\phi(x,y) + \xi x + \eta y) \, \mathrm dx\mathrm dy, \label{eq1.4}
\end{gather}
where $e_q(x) = e(x/q)$. Finally, we fix a smooth cutoff function $\Phi$ on $\mathbb R^d$ so that $\Phi(\bm\xi) = 1$ when $\max_j |\xi_j| \le \frac 18$ and $\Phi(\bm\xi) = 0$ when $\max_j |\xi_j| \ge \frac 14$. 

The next theorem states our asymptotic formula for the multiplier $\widehat{T_\lambda}(\bm\xi, \bm\eta)$. While we do not need this result directly in the proof of Theorem \ref{thm1}, such approximations are of independent interest: see \cite{ACHK, Cook19, Hugh17a, Magy02, MSW}. We include this theorem here, since its proof requires little work beyond what is needed to prove our main results. 

\begin{theorem}\label{thm2}
Let $d \ge 7$ and $\lambda \in \mathbb N$ be large. Then, for all $\bm\xi, \bm\eta \in \mathbb R^d$ and any fixed $\eps > 0$, one has
\begin{equation}\label{eqi.3}
\widehat{T_\lambda}(\bm\xi, \bm\eta) = \sum_{q=1}^{\infty} \sum_{\mathbf{m,n} \in \mathbb Z^d} G_\lambda(q; \mathbf{m,n})\Phi_q(\bm\xi_{q,\mathbf m})\Phi_q(\bm\eta_{q,\mathbf n}) I_\lambda(\bm\xi_{q,\mathbf m}, \bm\eta_{q,\mathbf n}) + O_\eps\big(\lambda^{-1/14 + \eps}\big),     
\end{equation} 
the series on the right being absolutely convergent. Here, $\Phi_q(\bm\xi) = \Phi(q\bm\xi)$, $\bm\xi_{q,\mathbf m} = \bm\xi - q^{-1}\mathbf m$, 
\begin{gather*}
G_\lambda(q; \mathbf{m,n}) = \sum_{\substack{1 \le \mathbf a \le q\\ (q,a_1,a_2,a_3) = 1}} e_q(-\lambda s(\mathbf a)) \prod_{j=1}^d g(q; \mathbf a, m_j, n_j), \\
I_\lambda(\bm\xi, \bm\eta) = \int_{\mathbb R^3} \bigg\{ \prod_{j=1}^d V_1 \big( \bm\beta; \lambda^{1/2} \xi_j, \lambda^{1/2} \eta_j \big) \bigg\} e(-s(\bm\beta)) \, \dbeta.
\end{gather*}
Moreover, if $\lambda$ is even, one has
\[ 1 \lesssim \widehat{T_\lambda}(\bm0, \bm0) \lesssim 1. \]
\end{theorem}

Before we state our main approximation to $T_\lambda$, we pause for a moment to observe that since
\begin{equation}\label{eqi.4}
T^*(f,g) \leq \|g\|_{\infty}\cdot T^*(|f|,1),
\end{equation}
we may, for the proof of Theorem \ref{thm1}, focus on the restriction of $T_\lambda$ to its first argument, 
\[ T_\lambda f = T_\lambda(f,1). \]
In particular, we establish our main approximation formula, given by the next theorem, for $T_\lambda f$ only.

\begin{theorem}[Approximation formula]\label{thm3}
Let $d\geq 9$ and $p > p_0(d)$. When $\lambda\in\mathbb{N}$, one has 
\begin{equation}\label{eqi.5} 
T_\lambda f = M_\lambda f + E_\lambda f, 
\end{equation}
where: 
\begin{itemize}
\item [(i)] $M_\lambda$ is the convolution operator with Fourier multiplier 
\[ \widehat{M_\lambda}(\bm\xi) = c_d\sum_{q=1}^{\infty} \sum_{\mathbf m \in \mathbb Z^d} G_\lambda(q; \mathbf{m,0})\Phi(q\bm\xi - \mathbf m) \widetilde{\mathrm dS}\big( \lambda^{1/2}(\bm\xi - q^{-1}\mathbf m) \big), \] 
with $c_d > 0$ and $\widetilde{\mathrm dS}(\bm\xi)$ being the Fourier transform of the Euclidean surface measure on the unit sphere in $\mathbb R^d$ (see \eqref{eq2.16} below).
\item [(ii)] There exists an exponent $\delta_p = \delta_p(d) > 0$ such that the error term operator $E_\lambda$ satisfies the maximal inequality
\begin{equation}\label{eqi.6}
\Big\| \sup_{\lambda \in [\Lambda/2,\Lambda)}|E_\lambda f| \Big\|_p \lesssim_\eps \Lambda^{-\delta_p + \eps} \| f \|_p
\end{equation}
for any fixed $\eps > 0$; in particular, one can choose $\delta_2 = \min\big( \frac 14, \frac 1{8}(d-8) \big)$. 
\end{itemize}
\end{theorem}

In view of \eqref{eqi.4}, Theorem \ref{thm1} is a direct consequence of Theorem \ref{thm3} and Proposition \ref{p3} below, which establishes the boundedness on ${\ell^p(\mathbb Z^d)}$ of the maximal operator
\[ M^*f = \sup_{\lambda \in \mathbb N} |M_\lambda f|. \]

The outline of the remainder of the paper is as follows. In Section \ref{s2}, we demonstrate several technical lemmas, mostly from number theory, to be used later on. Section \ref{s3} provides an outline of the proof of Theorem \ref{thm3}, breaking it up into several key propositions. The key idea there is the application of the Hardy--Littlewood circle method to decompose the operator $T_\lambda$ and its Fourier multiplier into major and minor arc contributions. We analyze those contributions separately in Sections~\ref{s4} and \ref{s5}, using the results developed in Section \ref{s2} as well as various new techniques described therein. In Section \ref{s6}, we sketch the proof of Theorem~\ref{thm2}. Since that proof tracks closely the proof of Theorem \ref{thm3}, we focus primarily on explaining the necessary modifications. Section \ref{s7} contains some remarks on connections between our results and questions about the distribution of equilateral triangles with vertices in $\mathbb Z^d$. We close the paper, in Section \ref{s8}, with some discussion in support of the conjecture we made above that the optimal ranges for $d$ and $p$ in Theorem \ref{thm1} should be $d \ge 7$ and $p > d/(d-3)$. In particular, we demonstrate that a hypothetical bound for the exponential sum $S_N(\bm\alpha; \xi, \eta)$ below will yield the conclusions of Theorems \ref{thm1} and \ref{thm3} for $d \ge 7$ and $p > d/(d-3)$.

\subsection*{Acknowledgments}
The first author was supported in part by NSF grant DMS-1502464. The second author thanks Towson University for sabbatical support that allowed this work to be completed. The third author was supported in part by Simons Foundation Grant \#360560. Last but not least, the first two authors thank Trevor Wooley for several helpful discussions and generous advice.
 
\section{Background material}
\label{s2}

Most of the work in this section concerns the analysis of the exponential sum
\[ S_N(\bm\alpha; \xi, \eta) = \sum_{|x| \leq N} \sum_{|y| \le N} e(\bm\alpha \cdot \bm\phi(x,y) + \xi x + \eta y), \]
which is the cornerstone of our application of the circle method. 

The first two lemmas provide bounds for the exponential sum $g(q; \mathbf a, m, n)$ defined in \eqref{eq1.3}. Henceforth, we abbreviate $\gcd(a,b,\dots)$ and $\lcm[a,b,\dots]$ as $(a,b,\dots)$ and $[a,b,\dots]$, respectively. 

\begin{lemma}\label{l1}
Suppose that $(q,a_1,a_2,a_3) = 1$. Then 
\[ |g(q; \mathbf a, m, n)| \lesssim q^{-1}(q, a_1a_3-a_2^2)^{1/2} =: q^{-1}w_q(\mathbf a). \]
\end{lemma}

\begin{proof}
We have
\[ q^4|g(q; \mathbf a, m, n)|^2 = \sum_{h,k = 1}^q \sum_{x,y=1}^q e_q(F(x,y,h,k)), \]
where
\begin{align*}
F(x,y,h,k) &= a_1h(2x+h) + 2a_2(xk + yh + hk) + a_3k(2y+k) + mh + nk\\
&= 2x(a_1h+a_2k) + 2y(a_2h + a_3k) + f(h,k), \quad \text{say.}
\end{align*} 
Thus,
\begin{align} 
q^4|g(q; \mathbf a, m, n)|^2 &= \sum_{h,k = 1}^q e_q \big( f(h,k) \big) \sum_{x,y=1}^q e_q\big( 2x(a_1h+a_2k) + 2y(a_2h + a_3k) \big) \notag\\
&\le q^2\nu(q; 2\mathbf a), \label{eql1.1}
\end{align} 
where $\nu(q; \mathbf a)$ denote the number of solutions $(h,k) \in \mathbb Z_q^2$ of the pair of congruences
\begin{equation}\label{eql1.2}
a_1h + a_2k \equiv a_2h + a_3k \equiv 0 \pmod q.    
\end{equation}  
The arithmetic function $\nu(q; \mathbf a)$ is multiplicative in $q$ and satisfies
\[ \nu(q; 2\mathbf a) = \begin{cases} \nu(q; \mathbf a) & \text{if } q \text{ is odd}, \\
4\nu(q/2; \mathbf a) & \text{if } q \text{ is even}.
\end{cases} \]
Therefore, the lemma will follow from \eqref{eql1.1}, if we show that, for $p^r \mid q$,
\begin{equation}\label{eql1.3} 
\nu(p^r; \mathbf a) \le (p^r, a_1a_3-a_2^2).     
\end{equation}  

Consider \eqref{eql1.2} with $q=p^r$ and write $p^s = (p^r, a_1a_3-a_2^2)$. By hypothesis, we have $(a_i,p)=1$ for some $1 \le i \le 3$: say, $(a_1, p) = 1$. Let  $\overline{a_1}$ denote the multiplicative inverse of $a_1$ modulo~$p^r$. Then \eqref{eql1.2} gives
\[ h \equiv -\overline{a_1}a_2k \pmod{p^r}, \qquad (a_1a_3-a_2^2)k \equiv 0 \pmod {p^r}. \]
The latter congruence determines $k$ modulo $p^{r-s}$, so there are $p^s$ possibilities for $k$; and for each of those $k$, there is a single choice for $h$. Hence, \eqref{eql1.2} with $q = p^r$ has exactly $p^s$ solutions. This establishes \eqref{eql1.3}.
\end{proof}

\begin{lemma}\label{l2}
Let $w_q(\mathbf a)$ be the function appearing in the statement of Lemma \ref{l1}. Then, for real $s \ge 2$, one has  
\begin{equation}\label{eql2.1}
\sum_{\substack{1 \le \mathbf a \le q\\ (q, a_1, a_2, a_3) = 1}} w_q(\mathbf a)^s \le \tau(q)^2 q^{s/2 + 2},     
\end{equation} 
where $\tau(q)$ is the number of positive divisors of $q$.
\end{lemma}

\begin{proof}
Both sides of \eqref{eql2.1} are multiplicative in $q$, so it suffices to consider the case $q = p^m$, with $p$ prime. In this case, the left side of \eqref{eql2.1} becomes
\[ \sum_{k = 0}^m p^{sk/2} \nu(p^m; k), \]
where $\nu(p^m; k)$ is the number of triples $a_1, a_2, a_3$ with
\begin{equation}\label{eql2.2} 
1 \le a_i \le p^m, \quad (p, a_1, a_2, a_3) = 1, \quad (p^m, a_1a_3 - a_2^2) = p^k. 
\end{equation} 

Let $k > 0$ and suppose that $(p,a_1) = 1$. Then the congruence 
\[ a_1a_3 \equiv a_2^2 \pmod{p^k}, \] 
which is implicit in \eqref{eql2.2}, has $p^{m-k}$ solutions $a_3$ for each choice of $a_1, a_2$. By symmetry, a similar conclusion holds also when $(p,a_3)=1$. Hence,
\[ \nu(p^m; k) \le 2p^{3m-k} + \nu_0(p^m; k), \]
where $\nu_0(p^m; k)$ is the number of triples $a_1, a_2, a_3$ with
\[ 1 \le a_i \le p^m, \quad (p, a_1, a_3) = p, \quad (p, a_2) = 1, \quad (p^m, a_1a_3 - a_2^2) = p^k. \]
Since the last three conditions are inconsistent when $k > 0$, we conclude that
\begin{equation}\label{eql2.3} 
\nu(p^m; k) \le 2p^{3m-k}. 
\end{equation} 

Combining \eqref{eql2.3} and the trivial bound $\nu(p^m;0) \le p^{3m}$, we deduce that
\begin{align*} 
\sum_{k = 0}^m p^{sk/2} \nu(p^m; k) &\le p^{3m} + 2p^{3m}\sum_{k=1}^m p^{k(s/2-1)} \\
&\le p^{3m} + 2mp^{m(s/2+2)} < \tau(p^m)^2p^{m(s/2+2)}. \qedhere 
\end{align*}
\end{proof}

The next lemma bounds for the exponential integral $V_N(\bm\beta; \xi, \eta)$ defined in \eqref{eq1.4}. It is an immediate consequence of Theorem 1.5 in \cite{ACK04}.

\begin{lemma}\label{l3}
One has
\[ |V_N(\bm\beta; \xi, \eta)| \lesssim N^2\Delta(1 + N^2|\bm\beta| + N|\xi| + N|\eta|), \]
where $\Delta(x) = x^{-1/2}\log(x+1)$.
\end{lemma}

For our analysis of the exponential sum $S_N(\bm\alpha; \xi, \eta)$, we need to define sets of major and minor arcs. When $1 \le P \le N$ and $a,q \in \mathbb N$ with $1 \le a \le q \le P$, we define the one-dimensional major arc ${\mathfrak M}(q,a)$ as the closed interval
\begin{equation}\label{eq2.7}  
{\mathfrak M}(q,a) = {\mathfrak M}(N,P; q,a) = \left[ \frac aq - \frac P{qN^{2}},\frac aq + \frac P{qN^{2}} \right].     
\end{equation} 
Major arcs in $\mathbb T^3$ are then defined as Cartesian products of single-dimensional ones in two ways. Given a rational point
\[ \mathbf r = \left( \frac {a_1}{q_1}, \frac {a_2}{q_2}, \frac {a_3}{q_3} \right) = \left( \frac {b_1}{q}, \frac {b_2}{q}, \frac {b_3}{q} \right), \]
where
\[ (a_1,q_1) = (a_2,q_2) = (a_3, q_3) = (b_1,b_2,b_3,q) = 1, \]
we consider two major arcs centered at $\mathbf r$:
\[ \mathfrak N(\mathbf{q;a}) = {\mathfrak M}(q_1,a_1) \times {\mathfrak M}(q_2,a_2) \times {\mathfrak M}(q_3,a_3) \]
and
\[ \mathfrak M(q; \mathbf{b}) = {\mathfrak M}(q,b_1) \times {\mathfrak M}(q,b_2) \times {\mathfrak M}(q,b_3). \]
We then define the respective sets of major and minor arcs as
\begin{equation}\label{eq2.8n} 
 \mathfrak N = \mathfrak N(P) = \bigcup_{\substack{1 \le \mathbf a \le \mathbf q \le P\\ (a_i,q_i) = 1}} \mathfrak N(\mathbf{q; a}), \quad \mathfrak n = \mathfrak n(P) = \mathbb T^3 \setminus {\mathfrak N},     
\end{equation} 
and
\begin{equation}\label{eq2.8m} 
 \mathfrak M = \mathfrak M(P) = \bigcup_{\substack{1 \le \mathbf a \le q \le P\\ (a_1, a_2, a_3, q) = 1}} \mathfrak M(q; \mathbf{a}), \quad \mathfrak m = \mathfrak m(P) = \mathbb T^3 \setminus \mathfrak M.     
\end{equation} 

When $\bm\alpha$ is in the set of minor arcs $\mathfrak n(P)$, we bound $S_N(\bm\alpha; \xi, \eta)$ using the following lemma.

\begin{lemma}\label{l5}
Let the set of minor arcs $\mathfrak n = \mathfrak n(P)$ be given by \eqref{eq2.8n} with $1 \le P \le N$. Then for all $\xi,\eta \in \mathbb T$,
\[ \sup_{\bm\alpha \in \mathfrak n} |S_N(\bm\alpha; \xi, \eta)| \lesssim_\eps N^{2+\eps}P^{-1/2}. \]
\end{lemma}

\begin{proof}
By Dirichlet's theorem on Diophantine approximation, there exist rational approximations $a_i/q_i$, $i=1,2,3$, such that
\begin{equation}\label{eql5.1}
|q_i\alpha_i - a_i| \le PN^{-2}, \quad (a_i, q_i) = 1, \quad 1 \le q_i \le N^2P^{-1}.    
\end{equation}  
Due to our assumption that $\bm\alpha \in \mathfrak n$, we must have $q_i > P$ for at least one index $i$, and by symmetry, we may assume that $i=1$ or $2$. 

We have
\begin{align*}
|S_N(\bm\alpha; \xi, \eta)|^2 & \le N \sum_{|y| \le N} \bigg| \sum_{|x| \le N} e\big( \alpha_1x^2 + 2\alpha_2xy + \xi x\big) \bigg|^2 \\   
& \le N \sum_{|y| \le N} \sum_{|h| \le 2N} \sum_{x \in I(h)} e\big( \alpha_1h(2x+h) + 2\alpha_2hy + \xi h\big) \\   
& \le N  \sum_{|k| \le 4N} \bigg| \sum_{x \in I(k/2)} e( \alpha_1kx ) \bigg| \cdot \bigg| \sum_{|y| \le N} e(\alpha_2ky) \bigg| \\   
& \lesssim N^3 + N \sum_{k \le 4N} \prod_{j=1}^2 \min \big( N, \| \alpha_j k \|^{-1} \big),   
\end{align*}  
where $I(h)$ is a subinterval of $[-N,N]$ that depends on $h$ and $\|x\| = \min\{ |x-n| : n \in \mathbb Z \}$. Since we have \eqref{eql5.1}, we can now apply Lemma 2.2 in Vaughan \cite{Vaug97} to deduce that, for $i=1,2$, 
\[ |S_N(\bm\alpha; \xi, \eta)|^2 \lesssim N^{4} \big( q_i^{-1}+ N^{-1} + q_iN^{-2} \big) \log N. \] 
The lemma follows on recalling that $P < q_i \le N^2/P$ for at least one of $i=1$ or $2$.
\end{proof}

Next, we establish a local approximation for $S_N(\bm\alpha; \xi, \eta)$ when $\bm\alpha$ is on a major arc $\mathfrak M(q; \mathbf{a})$. 

\begin{lemma}\label{l4}
Let $\bm\alpha \in \mathbb T^3$, $\xi,\eta \in \mathbb T$, $q \in \mathbb N$, $\mathbf a \in \mathbb Z^3$ with $(q,a_1,a_2,a_3) = 1$, $m,n \in \mathbb Z$, and suppose that
\[ \left| \xi - \frac {m}q \right| \le \frac 1{2q}, \quad \left| \eta - \frac {n}q \right| \le \frac 1{2q}. \]
Then 
\[ S_N(\bm\alpha; \xi, \eta) = g(q; \mathbf a, m, n)V_N(\bm\beta; \theta_1, \theta_2) + O \big( qN( 1 + N^2|\bm\beta|) \big), \] 
where $\bm\beta = \bm\alpha - q^{-1}\mathbf a$, $\theta_1 = \xi - m/q$, $\theta_2 = \eta - n/q$.
\end{lemma}

\begin{proof}
The result follows by partial summation from the asymptotic formula
\begin{equation}\label{eql4.1}
\sum_{\substack{X < n \le Y\\ n \equiv a \; (\mathrm{mod} \, q)}} e(\theta n) = \frac 1q \int_X^Y e(\theta x) \, \mathrm dx + O(1), 
\end{equation}
where $a,q \in \mathbb N$ and $|\theta| \le (2q)^{-1}$. 

Let
\[ S_{r,s}(\bm\alpha; \xi, \eta) = \sum_{\substack{|x| \le N\\x \equiv r \; (\mathrm{mod} \, q)}} \sum_{\substack{|y| \le N\\y \equiv s \; (\mathrm{mod} \, q)}} e\big( \bm\alpha \cdot \bm\phi(x,y) + \xi x + \eta y \big). \]
By splitting the terms in $S_N(\bm\alpha; \xi, \eta)$ according to their residues modulo $q$, we get
\begin{align}\label{eql4.2}
S_N(\bm\alpha; \xi, \eta) &= \sum_{r,s=1}^{q}  S_{r,s}(\bm\alpha; \xi, \eta) \notag\\
&= \sum_{r,s=1}^{q} e_q( \mathbf a \cdot \bm\phi(r,s) + mr + ns)  S_{r,s}(\bm\beta; \theta_1, \theta_2).
\end{align}
For a fixed $x$, $|x| \le N$, partial summation over $y$ and \eqref{eql4.1} yield
\[ \sum_{\substack{|y| \le N\\y \equiv s \; (\mathrm{mod} \, q)}} e\big( 2\beta_2xy + \beta_3y^2 + \theta_2y \big) = \frac 1q \int_{-N}^N e\big( 2\beta_2xy + \beta_3y^2 + \theta_2y \big) \, \mathrm dy + O\big( 1 + N^2|\bm\beta| \big). \]
Similarly, for a fixed $y$, $|y| \le N$, we get
\[ \sum_{\substack{|x| \le N\\x \equiv r \; (\mathrm{mod} \, q)}} e\big( \beta_1x^2+ 2\beta_2xy + \theta_1x \big) = \frac 1q \int_{-N}^N e\big( \beta_1x^2+ 2\beta_2xy + \theta_1x \big) \, \mathrm dx + O\big( 1 + N^2|\bm\beta| \big). \]
Together, these two approximations give
\[ S_{r,s}(\bm\beta; \theta_1, \theta_2) = q^{-2} V_N(\bm\beta; \theta_1, \theta_2) + O\big( q^{-1}N(1 + N^2|\bm\beta|) \big). \]
The claim of the lemma follows from this approximation and \eqref{eql4.2}. 
\end{proof}

When $P$ is not too large, Lemmas \ref{l5} and \ref{l4} can be combined to extend the bound of Lemma~\ref{l5} to the wider set of minor arcs $\mathfrak m(P)$ defined by \eqref{eq2.8m}. The next lemma provides the details. 

\begin{lemma}\label{l6}
Let the set of minor arcs $\mathfrak m = \mathfrak m(P)$ be given by \eqref{eq2.8m} with $1 \le P \le N$. Then for all $\xi,\eta \in \mathbb T$,
\begin{equation}\label{eq2.12} 
\sup_{\bm\alpha \in \mathfrak m} |S_N(\bm\alpha; \xi, \eta)| \lesssim_\eps N^{2+\eps}P^{-1/2} + NP^3.      
\end{equation} 
In particular, if $1 \le P \le N^{2/7}$, 
\[ \sup_{\bm\alpha \in \mathfrak m} |S_N(\bm\alpha; \xi, \eta)| \lesssim_\eps N^{2+\eps}P^{-1/2}. \]
\end{lemma}

\begin{proof}
Let $\mathfrak N$ and $\mathfrak n$ be the sets of major and minor arcs defined by \eqref{eq2.8n}. When $\bm\alpha \in \mathfrak n$, the bound \eqref{eq2.12} follows from Lemma \ref{l5}, so we may focus on the case when $\bm\alpha \in \mathfrak N \cap \mathfrak m$. Suppose that $\bm\alpha \in \mathfrak N(\mathbf{q;a})$ and define 
\[ b_i = \frac {a_iq}{q_i}, \quad \beta_i = \alpha_i - \frac {a_i}{q_i} = \alpha_i - \frac {b_i}q  \qquad (1 \le i \le 3), \] 
where $q = [q_1, q_2, q_3]$. We remark that since $(a_i,q_i)=1$ for all $i$, one has $(q, b_1, b_2, b_3) = 1$ and $\mathfrak M(q; \mathbf b) \subseteq \mathfrak N(\mathbf{q;a})$. Since $\bm\alpha \in \mathfrak m$, we must have 
\[ q \ge P \quad \text{ or } \quad \bm\alpha \in \mathfrak N(\mathbf{q;a}) \setminus \mathfrak M(q; \mathbf b), \]
and hence,
\begin{equation}\label{eql6.0} 
(q + qN^2|\bm\beta|)^{-1/2} \le P^{-1/2}. 
\end{equation} 

Choose integers $m,n$ such that 
\[ |q\xi - m| \le \frac 12, \quad |q\eta - n| \le \frac 12. \] 
Lemma \ref{l4} gives 
\begin{equation}\label{eql6.1}
S_N(\bm\alpha; \xi, \eta) =  g(q; \mathbf b, m, n)V_N( \bm\beta; \theta_1, \theta_2) + O \big( qN( 1 + N^2|\bm\beta|) \big), 
\end{equation} 
where $|\theta_i| \le (2q)^{-1}$. We have
\[ q(1 + N^2|\bm\beta|) \le q_1q_2q_3(1 + Pq_1^{-1} + Pq_2^{-1} + Pq_3^{-1}) \lesssim P^3. \]
We now use Lemmas \ref{l1} and \ref{l3} to bound the main term in the approximation \eqref{eql6.1} and obtain
\begin{equation}\label{eql6.2} 
S_N(\bm\alpha; \xi, \eta) \lesssim_\eps N^{2+\eps}(q + qN^2|\bm\beta|)^{-1/2} + NP^3. 
\end{equation} 
The lemma follows from \eqref{eql6.2} and \eqref{eql6.0}. 
\end{proof}

Our next lemma provides an upper bound for $S_N(\bm\alpha; \xi, 0)$ for $\bm\alpha \in \mathfrak N(\mathbf{q;a})$ that is stronger than \eqref{eql6.2} above. It leads immediately to a stronger version of Lemma \ref{l6} in the special case $\eta = 0$. Without this result, the range of $p$ in Theorems \ref{thm1} and \ref{thm3} would be significantly reduced. The proof of this lemma is too technical to include here and will appear as a part of a forthcoming work of the second author \cite{Kumc20}. It is based on the results of Vaughan \cite{Vaug09} on quadratic Weyl sums and uses a more sophisticated version of the ideas behind the proof of Proposition \ref{p1} in Section \ref{s4} below (in particular, see \eqref{eqN.3}).

\begin{lemma}\label{l4++}
Let $1 \le P \le 0.1N^{1/2}$ and let $\mathfrak N = \mathfrak N(P)$ be the set of major arcs given by \eqref{eq2.8n}. Then for all $\bm\alpha \in \mathfrak N(\mathbf{q;a})$ and $\xi \in \mathbb T$, one has
\[ |S_N(\bm\alpha; \xi, 0)| \lesssim_\eps \frac{N^{2+\eps}} {(q + qN^2|\bm\beta|)^{1/2}} + NP^{1/2+\eps}, \]
where $q = [q_1,q_2,q_3]$ and $\beta_i = \alpha_i - a_i/q_i$. Moreover, if $\mathfrak m = \mathfrak m(P)$ is the respective set of minor arcs given by \eqref{eq2.8m}, one has
\[ \sup_{\bm\alpha \in \mathfrak m}|S_N(\bm\alpha; \xi, 0)| \lesssim_\eps N^{2+\eps}P^{-1/2}. \]
\end{lemma}

The next lemma is Theorem 2.1 of Bourgain and Demeter \cite{BoDe16}.

\begin{lemma}\label{l7}
For $s \ge 1$, let $J_{s,2,2}(N)$ denote the number of solutions of the system 
\[ \sum_{i=1}^s x_i^ky_i^l = \sum_{i=s+1}^{2s} x_i^ky_i^l \qquad (1 \le k+l \le 2; \; k,l \ge 0), \]
in integers  $x_1, y_1, \dots, x_s, y_s \in [-N, N]$. Then, for every fixed $\eps > 0$, one has 
\[ J_{s,2,2}(N) \lesssim_\eps N^{2s+\eps} + N^{4s-8+\eps}. \]
\end{lemma}

Lemma \ref{l7} is, in fact, a bound for the $2s$-th moment of the exponential sum $S_N(\bm\alpha; \xi, \eta)$ where we average over all five arguments. The next lemma provides an alternative bound for the sixth moment of $S_N(\bm\alpha; \xi, \eta)$ where we average only over $\bm\alpha$.

\begin{lemma}\label{l8}
For all $\xi,\eta \in \mathbb T$ and any fixed $\eps > 0$, one has 
\[ \int_{\mathbb T^3} |S_N(\bm\alpha; \xi, \eta)|^6 \, \dalpha \lesssim_\eps N^{6+\eps}. \]
\end{lemma}

\begin{proof}
The given integral is bounded above by the number of solutions of the system 
\[ x_1^ky_1^l + x_2^ky_2^l + x_3^ky_3^l = x_4^ky_4^l + x_5^ky_5^l + x_6^ky_6^l \qquad (k+l = 2; \; k,l \ge 0) \]
in integers  $x_1, y_1, \dots, x_6, y_6 \in [-N, N]$. We denote this quantity by $T(N)$. Also, for $a, c \in \mathbb N$ and $b \in \mathbb Z$, let 
\[ \nu(a,b,c) = \#\big\{ \mathbf x, \mathbf y \in \mathbb Z^3 : |\mathbf x|^2 = a, \; |\mathbf y|^2 = c, \; \mathbf x \cdot \mathbf y = b \big\}. \] 
We have
\begin{equation}\label{eql8.1}
T(N) \le \sum_{0 \le a,c \le X} \sum_{|b| \le X} \nu(a,b,c)^2, 
\end{equation}
where $X = 3N^2$. 

By the Cauchy--Schwarz inequality, $\nu(a,b,c)$ is positive only if $b^2 \le ac$. When $b^2 - ac < 0$, Corollary 1.3 in a recent preprint of Bourgain and Demeter \cite{BoDe18} gives 
\[ \nu(a,b,c) \lesssim_\eps \gcd(a,b,c) (abc)^{\eps}. \]
From this, we deduce that
\begin{align}\label{eql8.2} 
\sum_{1 \le a,c \le X} \sum_{|b| < \sqrt{ac}} \nu(a,b,c)^2 &\lesssim_\eps X^{\eps} \sum_{1 \le a,c \le X} \sum_{|b| \le X} (a,b,c)^2 \notag\\
 &\lesssim_\eps X^{\eps} \sum_{d \le X} d^2 \bigg( \sum_{\substack{ 1 \le a,c \le X\\ d \mid a, d \mid c}}\sum_{\substack{ |b| \le X\\ d \mid b}} 1 \bigg) \notag\\
 &\lesssim_\eps X^{3+\eps} \sum_{d \le X} d^{-1} \lesssim_\eps X^{3+\eps}.
\end{align}

On the other hand, we have
\[ \sum_{|b| \le X} \sum_{\substack{ac = b^2\\0 \le a,c \le X}}  \nu(a,b,c)^2 \le \sum_{|b| \le X} \sum_{\substack{ac = b^2\\0 \le a,c \le X}}  r_3(a)^2 r_3(c)^2, \]
where $r_3(a)$ is the number of representations of $a$ as the sum of three squares. Using the bound 
\[ r_3(n) \lesssim_\eps n^{1/2+\eps} + 1, \]
we deduce that
\begin{align}\label{eql8.3} 
\sum_{|b| \le X} \sum_{\substack{ac = b^2\\0 \le a,c \le X}}  \nu(a,b,c)^2 &\lesssim_\eps \sum_{0 \le a \le X} a^{1+\eps} + \sum_{1 \le b \le X} \sum_{\substack{ ac = b^2\\ 1 \le a,c \le X}}(ac)^{1+\eps} \notag\\
&\lesssim_\eps X^{2+\eps} + X^{2+\eps} \sum_{b \le X} \tau(b^2) \lesssim_\eps X^{3+\eps}.
\end{align}
The lemma follows from \eqref{eql8.1}--\eqref{eql8.3}. 
\end{proof}

Let us define the integral 
\[ I_N(\lambda; \bm\xi) = \int_{\mathbb R^3} \bigg\{ \prod_{j=1}^d V_N(\bm\beta; \xi_j, 0) \bigg\} e(-\lambda s(\bm\beta)) \, \dbeta. \]
In the next lemma, we show that when $d \ge 7$ and $N^2 \ge \lambda$, its value is in fact independent of $N$ and can be expressed in terms of the Fourier transform of the surface measure on the unit sphere in $\mathbb R^d$.

\begin{lemma}\label{l9}
When $d \ge 7$ and $N^2 \ge \lambda$, the singular integral $I_N(\lambda; \bm\xi)$ is absolutely convergent and satisfies
\[ I_N(\lambda; \bm\xi) = c_d\lambda^{d-3}\widetilde{\mathrm dS}(\lambda^{1/2}\bm\xi), \]
where $c_d > 0$ is a constant that depends only on the dimension and 
\begin{equation}\label{eq2.16}
\widetilde{\mathrm dS}(\bm\xi) = \int_{\mathbb S_{d-1}} e(\bm\xi \cdot \mathbf x) \, \mathrm dS
\end{equation}  
is the Fourier transform of the Euclidean surface measure on the unit sphere in $\mathbb R^d$. 
\end{lemma}

\begin{proof}
The absolute convergence of $I_N(\lambda; \bm\xi)$ follows from Lemma \ref{l3}, and a simple rescaling of the variables shows that
\[ I_N(\lambda; \bm\xi) = \lambda^{d-3}I_{N_\lambda}(1; \lambda^{1/2}\bm\xi), \]
where $N_\lambda = N\lambda^{-1/2}$. Hence, we may focus on $I_N(1; \bm\xi)$ with $N \ge 1$. Through the rest of this proof, we write $B_d$ for the $d$-dimensional Euclidean unit ball and $Q_d$ for the $d$-dimensional cube $[-N,N]^d$. We also define the polynomials
\[ f(\mathbf x) = 1-|\mathbf x|^2, \qquad g(\mathbf x, \mathbf y) = 2\mathbf x \cdot \mathbf y - 1. \]

We have
\[ I_N(1;\bm\xi) = \int_{Q_{d-1}} \int_{Q_{d-1}} \int_{\mathbb R} e(\beta_2 g( \mathbf x, \mathbf y)+\bm\xi' \cdot \mathbf x) U(\beta_2; \mathbf{x,y}) \, \mathrm d\beta_2 \, \mathrm d\mathbf x \, \mathrm d\mathbf y, \]
where $\bm\xi = (\bm\xi', \xi_d)$ and
\[ U(\beta_2; \mathbf{x,y}) = \int_{\mathbb R^2} V_N(\bm\beta; \xi_d, 0)e(-\beta_1 f(\mathbf x) - \beta_3f(\mathbf y)) \, \mathrm d\beta_1\mathrm d\beta_3. \]
We can rewrite the integral $V_N(\bm\beta; \xi, 0)$ as 
\[ V_N(\bm\beta; \xi, 0) = \int_0^{N^2} \int_0^{N^2} \cos( 4\pi \beta_2\sqrt{uv})\cos(2\pi\xi\sqrt u) e(\beta_1u + \beta_3v) \, \frac {\mathrm du\,\mathrm dv}{\sqrt{uv}}. \]
Hence, we can apply Fourier inversion to the integral over $\beta_1$ and $\beta_3$ to deduce that
\[ U(\beta_2; \mathbf{x,y}) = \frac{\cos\big( 4\pi \beta_2 \sqrt{f( \mathbf x) f(\mathbf y)} \big)\cos \big( 2\pi\xi_d \sqrt{f( \mathbf x)} \big)}{\sqrt{f( \mathbf x) f(\mathbf y)}}, \]
with $\mathbf{x,y}$ restricted to the set where $0 \le f(\mathbf x), f(\mathbf y) \le N^2$. The latter conditions restrict $\mathbf x$ and $\mathbf y$ to the unit ball $B_{d-1}$, which is a proper subset of their original domain $Q_{d-1}$ when $N \ge 1$. In particular, it becomes apparent that the parameter $N$ in the definition of $Q_d$ is superfluous as long as $N \ge 1$. Thus,
\[ I_N(1; \bm\xi) = I_1(1; \bm\xi) =: I(\bm\xi). \]

We now split the last coordinates of the variables $\mathbf{x,y}$: $\mathbf x = (\mathbf x', u)$ and $\mathbf y = (\mathbf y', v)$, with $u,v \in \mathbb R$. This allows us to rewrite $I(\bm\xi)$ once again in a different form, suitable for a subsequent application of Fourier inversion. Namely,
\begin{equation}\label{eql9.1}
I(\bm\xi) = \int_{B_{d-1}} \int_{B_{d-2}}  F(\mathbf x, \bm\xi) \int_{\mathbb R}  J(\mathbf x, \mathbf y', \beta)  e(\beta g(\mathbf x', \mathbf y')) \, \mathrm d\beta \, \frac{\mathrm d\mathbf x \, \mathrm d\mathbf y'}{\sqrt{f( \mathbf x)}}, 
\end{equation}
where
\begin{align*}
F(\mathbf x, \bm\xi) &= e(\bm\xi' \cdot \mathbf x)\cos \big(2\pi\xi_d \sqrt{f( \mathbf x)}\big), \\
J(\mathbf x, \mathbf y', \beta) &= 
\int_{-\sqrt{f(\mathbf y')}}^{\sqrt{f(\mathbf y')}} \frac {\cos\big( 4\pi \beta \sqrt{f( \mathbf x) (f(\mathbf y')-v^2)} \big)}{\sqrt{f(\mathbf y')-v^2}} e(2\beta uv) \, \mathrm dv \\ 
&= 
\int_{-1}^{1} \frac {\cos\big( 4\pi \beta \sqrt{f( \mathbf x)f(\mathbf y')(1-v^2)} \big)}{\sqrt{1-v^2}} e\big( 2\beta \sqrt{f(\mathbf y')} uv \big) \, \mathrm dv \\
&= 
\frac 12 \sum_{j \in \{1,2\}} \int_{-1}^{1} e\big( 2\beta \sqrt{f(\mathbf y')} \big( uv + (-1)^j\sqrt{f(\mathbf x)(1-v^2)} \big) \big) \, \frac {\mathrm dv}{\sqrt{1-v^2}}. 
\end{align*} 
Inserting this into \eqref{eql9.1} and rescaling $\beta$, we get
\begin{equation}\label{eql9.2}
I(\bm\xi) = \frac 14\int_{B_{d-1}} \int_{B_{d-2}}  \int_{\mathbb R}  \sum_{j \in \{1,2\}} K_j(\mathbf x, \theta)  e\bigg(\frac {\theta g(\mathbf x', \mathbf y')}{2\sqrt{f( \mathbf x')f(\mathbf y')}} \bigg) \, \mathrm d\theta \, \frac{F(\mathbf x, \bm\xi) \, \mathrm d\mathbf x \, \mathrm d\mathbf y'}{\sqrt{f( \mathbf x)f( \mathbf x')f( \mathbf y')}}, 
\end{equation}
where
\[ K_j(\mathbf x, \theta) = \int_{-1}^{1} e \bigg( \frac{\theta}{\sqrt{f(\mathbf x')}} \big( uv + (-1)^j\sqrt{f(\mathbf x)(1-v^2)} \big) \bigg) \, \frac {\mathrm dv}{\sqrt{1-v^2}}. \]

Define 
\[ \alpha_{\mathbf x} = \arcsin \bigg(  \frac{u}{\sqrt{f(\mathbf x')}} \bigg), \quad \beta_{j,\mathbf x} = (-1)^{j}\alpha_{\mathbf x}. \]
After some obvious changes of the variables, we find that, for $j=1,2$,
\begin{align*}
K_j(\mathbf x, \theta) &=  \int_{-\pi/2}^{\pi/2} e \big( \theta( \sin \phi\sin \alpha_{\mathbf x} + (-1)^j\cos \phi\cos \alpha_{\mathbf x} ) \big) \, \mathrm d\phi \\
&= \int_{-\pi/2}^{\pi/2} e \big( (-1)^{j}\theta \cos( \phi - \beta_{j,\mathbf x} ) \big) \, \mathrm d\phi \\
&= \int_{\beta_{j,\mathbf x}}^{\pi+\beta_{j,\mathbf x}} e \big( (-1)^{j}\theta \sin\phi \big) \, \mathrm d\phi = \bigg( \int_{\beta_{j,\mathbf x}}^{\pi/2} + \int_{-\beta_{j,\mathbf x}}^{\pi/2} \bigg) e \big( (-1)^{j}\theta \sin\phi \big) \, \mathrm d\phi \\
&= \bigg( \int_{\sin \alpha_{\mathbf x}}^1 + \int_{-\sin \alpha_{\mathbf x}}^{1} \bigg) e \big( (-1)^{j}\theta v \big) \, \frac{\mathrm dv}{\sqrt{1-v^2}} \\ 
&= \bigg( \int_{-1}^{-\sin \alpha_{\mathbf x}} + \int_{-1}^{\sin \alpha_{\mathbf x}} \bigg) e \big( (-1)^{j+1}\theta v \big) \, \frac{\mathrm dv}{\sqrt{1-v^2}}.
\end{align*}
Thus,
\[\sum_{j \in \{1,2\}} K_j(\mathbf x, \theta) =2 \int_{-1}^1 e(-\theta v) \, \frac{\mathrm dv}{\sqrt{1-v^2}}. \]
From this identity and \eqref{eql9.2}, we obtain by Fourier inversion that
\[ I(\bm\xi) = \int_{D}  \frac{F(\mathbf x, \bm\xi) \, \mathrm d\mathbf x \, \mathrm d\mathbf y'}{\sqrt{f( \mathbf x)\big(4f( \mathbf x')f( \mathbf y') - g(\mathbf x', \mathbf y')^2\big) }}, \]
where the domain of integration is the subset of $B_{d-1} \times B_{d-2}$ where
\[ |g(\mathbf x', \mathbf y')| \le 2\sqrt{f(\mathbf x')f(\mathbf y')}. \]

For a fixed $\mathbf x \in B_{d-1}$, the integral over $\mathbf y'$ can be expressed as
\[ G(\mathbf x) = \frac 1{2\sqrt{f(\mathbf x)f(\mathbf x')}} \int_{D_{\mathbf x'}} \frac {\mathrm d\mathbf y}{\sqrt{f(\mathbf y) - (\mathbf z \cdot \mathbf y - b)^2}}, \]
where 
\[ b = b(\mathbf x') = \frac 1{2\sqrt{f(\mathbf x')}}, \quad \mathbf z = \mathbf z(\mathbf x') = \frac {\mathbf x'}{\sqrt{f(\mathbf x')}}, \]
and $D_{\mathbf x'}$ is the $(d-2)$-dimensional ellipsoid defined by the inequality
\[ |\mathbf y|^2 + (\mathbf z \cdot \mathbf y - b)^2 \le 1. \]
Using basic algebra (repeated completion of the square) we can rewrite this inequality as
\[ \sum_{j=1}^{d-2} a_j^2(y_j + L_j(\mathbf y))^2 + \frac {b^2}{(a_1a_2 \cdots a_{d-2})^2} \le 1, \]
where $L_j(\mathbf y)$ is an affine function in the variables $y_{j+1}, \dots, y_{d-2}$ and $a_1, \dots, a_{d-2}$ are defined recursively by
\[ a_1 \cdots a_j = \sqrt{1 + z_1^2 + \dots + z_j^2}. \]
In particular, $ a_1 \cdots a_{d-2} = f(\mathbf x')^{-1/2}$. Hence,
\[ G(\mathbf x) = \frac 1{2\sqrt{f(\mathbf x)}} \int_{|\mathbf w| \le \sqrt 3/2} \frac {\mathrm d\mathbf w}{\sqrt{\frac 34 - |\mathbf w|^2}} =: \frac {2c_d}{\sqrt{f(\mathbf x)}}, \]
with a constant $c_d$ that depends only on the dimension. 

Finally, we note that $f(\mathbf x)^{-1/2}\,\mathrm d\mathbf x$, with $\mathbf x \in B_{d-1}$, is the standard surface measure on either the positive or negative hemisphere in $\mathbb R^d$. Hence,
\[ I(\bm\xi) = c_d \int_{\mathbb S_{d-1}}  e(\bm\xi \cdot \mathbf x) \, \mathrm dS. \]
\end{proof}

\section{Proof of Theorem \ref{thm3}}
\label{s3}

We assume that $p \le 2$. The starting point of our analysis is the observation that if $\lambda \le \Lambda$, one has
\begin{equation}\label{eq3.1} 
T_\lambda (f,g)(\mathbf x) = \lambda^{3-d} \int_{\mathbb T^3} F_N(\bm\alpha; f, g)(\mathbf x)e(-\lambda s(\bm\alpha)) \, \dalpha, 
\end{equation}
where $N = \Lambda^{1/2}$ and 
\[ F_N(\bm\alpha; f, g)(\mathbf x) = \sum_{ |\mathbf u| \le N} \sum_{|\mathbf v| \le N} e(\bm\alpha \cdot \bm\phi(\mathbf{ u,v})) f(\mathbf x - \mathbf{u})g( \mathbf x - \mathbf{v}). \]

We analyze the integral in \eqref{eq3.1} using the Hardy--Littlewood circle method, decomposing $\mathbb T^3$ into sets of major and minor arcs and estimating their respective contributions separately. For any measurable set $\mathfrak B \subset \mathbb T^3$, we write
\begin{equation}\label{eq3.2} 
T_\lambda (f; \mathfrak B) = \lambda^{3-d} \int_{\mathfrak B} F_N(\bm\alpha; f)e(-\lambda s(\bm\alpha)) \, \dalpha, 
\end{equation}
where $F_N(\bm\alpha; f) = F_N(\bm\alpha; f, 1)$. We introduce also the dyadic maximal functions of these operators: 
\[ T^*_{\Lambda, \mathfrak B}f = \sup_{\lambda \in [\Lambda/2,\Lambda)} |T_{\lambda}(f; \mathfrak B)|. \]

We define the major and minor arcs by \eqref{eq2.8m} with $P = 0.1N^{1/2}$ and obtain a decomposition of $T_\lambda$ as 
\begin{equation}\label{eq3.5}  
T_\lambda f = T_\lambda(f; {\mathfrak M}) + T_\lambda(f; {\mathfrak m}).  
\end{equation} 
The minor arc term on the right side of \eqref{eq3.5} is part of the error term $E_\lambda$ in \eqref{eqi.5}. In Section~\ref{s4}, we establish the following bound for its dyadic maximal function. 

\begin{prop}\label{p1}
Let $d \ge 9$ and $p_0(d) < p \le 2$. Then there exists an exponent $\alpha_p = \alpha_p(d) > 0$ such that
\begin{equation}\label{eq3.6}
\Big\|  \sup_{\lambda \in [\Lambda/2,\Lambda)} |T_{\lambda}(f; \mathfrak m)| \Big\|_{p} \lesssim_\eps \Lambda^{-\alpha_p+\eps}\| f \|_p
\end{equation}
for any fixed $\eps > 0$. In particular, we can choose $\alpha_2 = \frac 18(d-8)$.
\end{prop}

We now turn to $T_\lambda(f; \mathfrak M)$. Since the major arcs are disjoint, we deduce that
\[ T_\lambda(f; {\mathfrak M}) = \sum_{\substack{1 \le \mathbf a \le q \le P\\(q, a_1, a_2, a_3) = 1}} T_\lambda(f; {\mathfrak M}(q; \mathbf{a})) =: \sum_{q, \mathbf{a}} T_\lambda^{\mathbf{a}/q}f. \]
Thus, we may analyze the contribution of each individual major arc separately. When $\bm\alpha \in \mathfrak M(q; \mathbf{a})$, we develop a local approximation to the Fourier multiplier of $F_N(\bm\alpha; f)$. We use that approximation to guide our definition of an operator $M_\lambda^{\mathbf{a}/q}$, which provides a good approximation to $T_\lambda^{\mathbf{a}/q}$ for $\lambda \in [\Lambda/2, \Lambda)$. In Section~\ref{s5.1}, we establish the following proposition.

\begin{prop}\label{p2}
Let $d \geq 7$ and $q \le P$. Then, for any fixed $\eps > 0$, one has 
\begin{equation}\label{eq3.8}
\sum_{\substack{1 \le \mathbf a \le q\\(q, a_1, a_2, a_3)=1}} \Big\| \sup_{\lambda \in [\Lambda/2,\Lambda)} \big| \big(T_{\lambda}^{\mathbf{a}/q} - M_\lambda^{\mathbf{a}/q} \big) f \big| \Big\|_2 \lesssim_\eps q^{-1} \Lambda^{-\beta_2+\eps}\| f \|_2,
\end{equation}
where $\beta_2 = \beta_2(d) = \min( \frac 14, \frac 18(d-6))$. 
\end{prop}

In Section \ref{s5.2}, we study the operators $M_\lambda^{\mathbf{a}/q}$ further and show that, in fact,
\begin{equation}\label{eq3.9}
\sum_{q = 1}^{\infty} \sum_{\substack{1 \le \mathbf a \le q\\(q, a_1, a_2, a_3)=1}} M_\lambda^{\mathbf{a}/q} = M_\lambda, 
\end{equation}
where $M_\lambda$ is the operator defined in the statement of Theorem \ref{thm3}. We also establish the following result.

\begin{prop}\label{p3}
Let $d \ge 7$, $\frac d{d-1} < p \le 2$, and $q \in \mathbb N$. Then, for any fixed $\eps >0 $, one has 
\[ \sum_{\substack{1 \le \mathbf a \le q\\(q, a_1, a_2, a_3)=1}} \Big\| \sup_{\lambda \in \mathbb N} \big| M_\lambda^{\mathbf{a}/q} f \big| \Big\|_p \lesssim_\eps q^{-d/p'+2+\eps}\| f \|_p,\]
where $\frac 1{p'} = 1 - \frac 1p$. Consequently, the maximal operator
\[ M^*f = \sup_{\lambda \in \mathbb N} |M_\lambda f| \]
is bounded from $\ell^p(\mathbb Z^d)$ to $\ell^p(\mathbb Z^d)$ when $\frac d{d-3} < p \le 2$.
\end{prop}

Together, Propositions \ref{p1}--\ref{p3} suffice to establish the $\ell^2$-bound for the remainder term in the approximation formula. Indeed, combining Propositions \ref{p2} and \ref{p3}, we get
\begin{equation}\label{eq3.10}
\Big\| \sup_{\lambda \in [\Lambda/2,\Lambda)} \big| T_{\lambda}(f; \mathfrak M) - M_\lambda f \big| \Big\|_2 \lesssim_\eps \Lambda^{-\beta_2+\eps}\| f \|_2.
\end{equation}
To extend this to the full range of $p$ in the theorem, we interpolate between the case $p=2$ and a weaker bound, which we deduce from the following result on the dyadic maximal function $T_{\mathfrak M}^*$.

\begin{prop}\label{p4}
If $d \geq 7$ and $\frac{d}{d-3} < p \le 2$, one has
\[ \Big\|  \sup_{\lambda \in [\Lambda/2,\Lambda)}|T(f; {\mathfrak M})| \Big\|_{p} \lesssim \| f \|_p. \]
\end{prop}

We prove Proposition \ref{p4} in Section \ref{s5.3}. Here, we will use this proposition to complete the proof of Theorem \ref{thm3}. Observe that Propositions \ref{p3} and \ref{p4} give
\begin{equation}\label{eq3.11}
\Big\| \sup_{\lambda \in [\Lambda/2,\Lambda)} \big| T_{\lambda}(f; \mathfrak M) - M_\lambda f \big| \Big\|_p \lesssim \| f \|_p
\end{equation}
for all $p > d/(d-3)$. Thus, for any $r$ in the range $\frac d{d-3} < r < 2$, we can interpolate between \eqref{eq3.10} and the case $p = \frac d{d-3} + \eta$ of \eqref{eq3.11}, with $\eta > 0$ sufficiently small. We get
\[ \Big\| \sup_{\lambda \in [\Lambda/2,\Lambda)} \big| T_{\lambda}(f; \mathfrak M) - M_\lambda f \big| \Big\|_r \lesssim_\eps \Lambda^{-\theta(\beta_2 + \eps)} \| f \|_r, \] 
where $\theta$ is defined by 
\[ \frac 1r = \frac {\theta}2 + \frac {1-\theta}p. \]
In combination with Proposition \ref{p1}, this proves \eqref{eqi.6} with $\delta_p = \min( \alpha_p, \theta\beta_2)$. 
 
\section{Minor arc analysis}
\label{s4}

We begin our minor arc analysis with a reduction step that relates the operator norm of a maximal operator like $T^*_{\mathfrak m}$ to a mean value of an exponential sum. The reduction step uses the following variant of Lemma 7 in~\cite{ACHK}.  

\begin{lemma}\label{l10}
Let $X = \mathbb T^k$ or $\mathbb R^k$, for some $k \in \mathbb N$, and let $T_\lambda$, $\lambda \in \mathcal L$, be convolution operators on $\ell^2(\mathbb Z^d)$ with Fourier multipliers given by
\[ \widehat{T_\lambda}(\bm\xi) = \int_{X} K(\bm\alpha; \bm\xi) e(-\lambda\Phi(\bm\alpha)) \, \dalpha, \]
where $\Phi: X \to \mathbb R$ is continuous and $K( \cdot; \bm\xi) \in L^1(X)$ is a kernel independent of $\lambda$. Further, define the maximal function  
\[ T^* f (\mathbf x) = \sup_{\lambda} |T_\lambda f(\mathbf x)|. \]
Then 
\begin{equation}\label{eq4.1} 
\| T^*f \|_{2} \le \|f\|_2 \int_{X} \sup_{\bm\xi \in \mathbb T^d}  | K(\bm\alpha; \bm\xi) | \, \dalpha. 
\end{equation}
\end{lemma}

In the proof of Proposition \ref{p1}, we apply \eqref{eq4.1} with $X = \mathbb T^3$ and $K = \mathcal F_N \cdot \mathbf 1_{\mathfrak m}$, where
\begin{equation}\label{eq4.2}
\mathcal F_N(\bm\alpha; \bm\xi) = \prod_{j=1}^{d} S_N(\bm\alpha; \xi_j, 0).     
\end{equation} 
The supremum over $\bm{\xi}$ on the right side of \eqref{eq4.1} then stands in the way of a direct application of results from analytic number theory. Our next lemma overcomes this obstacle; its proof is similar to the proof of Lemma 3.2 in \cite{ACHK2}.

\begin{lemma}\label{l11}
If $s \in \mathbb N$ and $\mathfrak B \subseteq \mathbb T^3$ is a measurable set, then
\[ \int_{\mathfrak B} \sup_{\xi,\eta} |S_N(\bm\alpha; \xi, \eta) |^{2s} \, \dalpha \lesssim N^2 \int_{\mathfrak B} \int_{\mathbb T^2} |S_N(\bm\alpha; \xi, \eta) |^{2s} \, {\rm d}\xi{\rm d}\eta \, \dalpha.\]
\end{lemma}

\subsection*{Proof of Proposition \ref{p1}}

First, we consider the case $p=2$. We may assume that $\| f \|_2 = 1$. Lemma~\ref{l10} and the arithmetic-geometric mean inequality then give  
\begin{align}\label{eqp1.1}
\left\| T^*_{\mathfrak m}f \right\|_{2} &\lesssim \Lambda^{3-d} \int_{{\mathfrak m}} \sup_{\bm\xi \in \mathbb T^d} \left| \mathcal F_N(\bm\alpha; \bm\xi) \right| \, \dalpha \notag\\
&\lesssim \Lambda^{3 - d} \int_{{\mathfrak m}} \sup_{\xi \in \mathbb T} \left| S_N(\bm\alpha; \xi, 0) \right|^d \, \dalpha.
\end{align}

Observe that our choice of major and minor arcs is driven by Lemma \ref{l4++}: by setting $P$ to the maximum value permitted in that lemma, we have
\begin{equation}\label{eqp1.2} 
\sup_{\bm\alpha \in \mathfrak m} |S_N(\bm\alpha; \xi, 0)| \lesssim_\eps N^{7/4+\eps},  
\end{equation} 
for any fixed $\eps > 0$ and all $\xi \in \mathbb T$. We apply \eqref{eqp1.2} to all but eight copies of $S_N(\bm\alpha; \xi, 0)$ on the right side of \eqref{eqp1.1} and obtain
\[ \left\| T^*_{\mathfrak m}f \right\|_{2} \lesssim_\eps \Lambda^{-5}N^{(8-d)/4+\eps} \int_{\mathbb T^3} \sup_{\xi \in \mathbb T} \left| S_N(\bm\alpha; \xi, 0) \right|^8 \, \dalpha . \]
Lemma \ref{l11} now yields
\begin{align}\label{eqp1.3}
\left\| T^*_{\mathfrak m}f \right\|_{2} &\lesssim_\eps \Lambda^{-4}N^{(8-d)/4+\eps} \int_{{\mathbb T^3}} \int_{{\mathbb T^2}} \left| S_N(\bm\alpha; \xi, \eta) \right|^8 \,  {\rm d}\xi{\rm d}\eta \, \dalpha \notag\\
&\lesssim_\eps \Lambda^{-4}N^{(8-d)/4+\eps} J_{4,2,2}(N) \lesssim_\eps N^{(8-d)/4+2\eps},
\end{align}
by Lemma \ref{l7}. 

Next, we bound $T_{\mathfrak m}^*$ on $\ell^1(\mathbb Z^d)$. From \eqref{eq3.2}, we get
\[ \left\| T^*_{\mathfrak m} f \right\|_{1} \lesssim  \Lambda^{3-d} \| f \|_1 \int_{\mathfrak m} \bigg\{ \sum_{|x| \le N} \bigg| \sum_{|y| \le N} e(\alpha_3y^2 + 2\alpha_2 xy) \bigg| \bigg\}^d \, \dalpha. \]
If either $\alpha_2$ or $\alpha_3$ lies in the one-dimensional set of minor arcs $\mathfrak m(P)$, the proof of Lemma~\ref{l5} with the roles of $x$ and $y$ switched yields
\[  \sum_{|x| \le N} \bigg| \sum_{|y| \le N} e(\alpha_3y^2 + 2\alpha_2 xy) \bigg| \lesssim_\eps N^{2+\eps}P^{-1/2}. \]
Hence, by H\"older's inequality,
\begin{equation}\label{eqN.1}
\left\| T^*_{\mathfrak m} f \right\|_{1} \lesssim  \Lambda^{3} \| f \|_1 \bigg\{ P^{-d/2+\eps} + N^{-d-1} \int_{\mathfrak K} \sum_{|x| \le N} \bigg| \sum_{|y| \le N} e(\alpha_3y^2 + 2\alpha_2 xy) \bigg|^d \, \mathrm d\alpha_2\mathrm d\alpha_3 \bigg\},
\end{equation}
where $\mathfrak K$ are the two-dimensional major arcs
\[ \mathfrak K = \mathfrak M(P) \times \mathfrak M(P). \]

When $\alpha_3 = a_3/q_3 + \beta_3 \in \mathfrak M(q_3,a_3)$, with $(a_3,q_3) = 1$ and $1 \le q_3 \le P$, Theorem 8 of Vaughan \cite{Vaug09} gives
\[ \sum_{|y| \le N} e(\alpha_3y^2 + 2\alpha_2 xy) = \frac 1{q_3}\bigg( \sum_{r = 1}^{q_3} e_{q_3}\big( a_3r^2 + m_xr \big) \bigg) \int_{-N}^N e\big( \beta_3y^2 + \theta_x y \big) \, \mathrm dy + O(P^{1/2}), \]
where $m_x$ is the unique integer with 
\[ -\frac 12 \le 2q_3x\alpha_2 - m_x < \frac 12 \]
and $\theta_x = 2\alpha_2 x - m_x/q_3$. If $|\theta_x| \ge 3P/(q_3N)$, we deduce that
\[ \sum_{|y| \le N} e(\alpha_3y^2 + 2\alpha_2 xy) \lesssim_\eps NP^{-1/2}, \]
and so \eqref{eqN.1} yields
\begin{equation}\label{eqN.2}
\left\| T^*_{\mathfrak m} f \right\|_{1} \lesssim  \Lambda^{3} \| f \|_1 \bigg\{ P^{-d/2+\eps} + N^{-d-1} \int_{\mathfrak K} \; \sideset{}{^{(\bm\alpha)}}\sum_{|x| \le N} \bigg| \sum_{|y| \le N} e(\alpha_3y^2 + 2\alpha_2 xy) \bigg|^d \, \mathrm d\alpha_2\mathrm d\alpha_3 \bigg\},
\end{equation}
where the notation $\sum^{(\bm\alpha)}$ indicates that we are summing only over $x$ with $|\theta_x| < 3P/(q_3N)$. When $|x| \le N$, under the latter condition, we have
\begin{align*}
|q_2m_x - 2q_3xa_2| &\le q_2|m_x - 2q_3x\alpha_2| + 2q_3|x|\cdot|q_2\alpha_2-a_2| \\
&\le 3q_2PN^{-1} + 2q_3|x|PN^{-2} \le 5P^2N^{-1} < 1.
\end{align*} 
Therefore, 
\[ \frac {m_x}{q_3} = \frac {2xa_2}{q_2}. \]
We conclude that for those $\bm\alpha$ and $x$ that appear on the right side of \eqref{eqN.2}, Vaughan's approximation can be rewritten as
\[ \sum_{|y| \le N} e(\alpha_3y^2 + 2\alpha_2 xy) = \frac 1{q}\bigg( \sum_{r = 1}^{q} e_{q}\big( b_3r^2 + 2b_2xr \big) \bigg) \int_{-N}^N e\big( \beta_3y^2 + 2\beta_2xy \big) \, \mathrm dy + O(P^{1/2}), \]
where 
\[ q = [q_2,q_3], \quad b_i = \frac{a_iq}{q_i}, \quad \beta_2 = \alpha_2 - \frac {a_2}{q_2} = \alpha_2 - \frac {b_2}q. \]
Since $(q,b_2,b_3) = 1$, Theorems 7.1 and 7.3 in Vaughan \cite{Vaug97} now give 
\begin{align} 
\sum_{|y| \le N} e(\alpha_3y^2 + 2\alpha_2 xy) &\lesssim_\eps \bigg( \frac {(q/q_3,x)}{q} \bigg)^{1/2-\eps} \frac {N}{(1 + N|x\beta_2| + N^2|\beta_3|)^{1/2}} + P^{1/2} \notag\\
&\lesssim_\eps  \frac {q_3^{-1/2+\eps}N}{(1 + N^2|\beta_3|)^{1/2}} + P^{1/2}. \label{eqN.3}
\end{align}
Thus, the contribution of an individual major arc $\mathfrak K(\mathbf{q;a}) = \mathfrak M(q_2,a_2) \times \mathfrak M(q_3,a_3)$ to the right side of \eqref{eqN.2} is bounded above by
\begin{align*}
&\frac {P}{q_2q_3^{d/2}} \int_{\mathbb R} \frac {N^{d-1+\eps} \, \mathrm d\beta_3}{(1 +  N^2|\beta_3|)^{d/2}} + NP^{d/2}|\mathfrak K(\mathbf{q;a})| \lesssim_\eps \frac {PN^{d-3+\eps}}{q_2q_3^{d/2}}.    
\end{align*} 
Summing this bound over the different choices for $\mathbf{q,a}$, we deduce from \eqref{eqN.2} that
\begin{equation}\label{eqp1.4}
\left\| T^*_{\mathfrak m} f \right\|_{1} \lesssim_\eps  \Lambda^{3 - \kappa + \eps}\| f \|_1, \quad \kappa = \frac {\min(d,12)}{8}.
\end{equation}

Interpolating between \eqref{eqp1.3} and \eqref{eqp1.4}, we get \eqref{eq3.6} with 
\begin{align*}
\alpha_p &= \left( \frac{d-8}{8} \right)\left( 2 - \frac{2}{p} \right) + (3-\kappa)\left( 1 - \frac{2}{p} \right) \\
&= \frac {d + 4 - 4\kappa}4 - \frac {d + 16 - 8\kappa}{4p} > 0,   
\end{align*} 
provided that $p_0(d) < p \le 2$. \qed

\begin{remark}
Note that in the above argument we interpolate between a non-trivial $\ell^2$-bound and a \emph{non-trivial $\ell^1$-bound}. This appears to be a novel feature in our work that leads to a considerable strengthening of our main results. Indeed, the reader can easily check that if we use the trivial version of \eqref{eqp1.4} with $\kappa = 0$, we get Theorems \ref{thm1} and \ref{thm3} only for 
\[ p > \frac {d+16}{d+4}. \]
While the idea we use to get a non-trivial bound on $\ell^1$ is clearly dependent on the bilinearity of our operator, it is also quite general and should be applicable to other multilinear operators.  This idea of using the multilinearity to improve a certain linear estimate echoes a common theme in harmonic analysis, yet perhaps is new in the discrete setting.
\end{remark}

\section{The major arcs}
\label{s5}

\subsection{Proof of Proposition \ref{p2}}
\label{s5.1}

We fix a major arc $\mathfrak M(q; \mathbf{a})$ and a function $f \in \ell^2(\mathbb Z^d)$, with $\| f \|_2 = 1$. Also, we assume at first that $d \ge 9$. Recall that the Fourier multiplier of $T_\lambda^{\mathbf{a}/q}$ can be expressed as 
\[ \widehat{T_\lambda^{\mathbf{a}/q}}(\bm\xi) =
\lambda^{3-d} e_q(-\lambda s(\mathbf a)) \int_{\mathfrak M_q} \mathcal F_N(q^{-1}\mathbf a + \bm\beta; \bm\xi) e(-\lambda s(\bm\beta)) \, \dbeta, \]
where 
\[ \mathfrak M_q = \mathfrak M(q; \mathbf{a}) -  q^{-1}\mathbf a. \] 

Lemma \ref{l4} suggests that a convolution with the following multiplier should define a good approximation to $T_\lambda^{\mathbf{a}/q}$:
\[ \widehat{A_\lambda^{\mathbf{a}/q}}(\bm\xi) = \lambda^{3-d} e_q(-\lambda s(\mathbf a)) \int_{\mathfrak M_q} \mathcal{G}_N(\bm\beta, q, \mathbf{a}; \bm\xi) e(-\lambda s(\bm\beta)) \, \dbeta, \]
where 
\[ \mathcal{G}_N(\bm\beta, q, \mathbf{a}; \bm\xi) = g(q; \mathbf {a, m}) V_N( \bm\beta; \bm\xi_{q,\mathbf m}), \]
with $m_j = \lfloor q\xi_j + \frac 12 \rfloor$, $\bm\xi_{q,\mathbf m} = \bm\xi - q^{-1}\mathbf m$, and
\[ g(q; \mathbf{a,m}) = \prod_{j=1}^d g(q; \mathbf a, m_j, 0), \quad
V_N( \bm\beta; \bm\xi) = \prod_{j=1}^d V_N( \bm\beta; \xi_j, 0). \]
Let $A_\lambda^{\mathbf{a}/q}$ denote the convolution operator with this Fourier multiplier. 

Similarly to \eqref{eql6.1} and \eqref{eql6.2} in the proof of Lemma \ref{l6} (but using the full strength of Lemma \ref{l1} this time), we find that
\begin{equation}\label{eqp2.1}
S_N(q^{-1}\mathbf a + \bm\beta; \xi, \eta) \lesssim_\eps \tilde w_q N^{2+\eps} \Psi(\bm\beta)^{-1/2} + N\Psi(\bm\beta),
\end{equation} 
where 
\begin{equation}\label{eqp2.2} 
\Psi(\bm\beta) = q(1 + N^2|\bm\beta|), \quad \tilde w_q = q^{-1/2}w_q(\mathbf a), 
\end{equation} 
$w_q(\mathbf a)$ being the function that appears in Lemmas \ref{l1} and \ref{l2}.

When $\bm\beta \in \mathfrak M_q$, we have $\Psi(\bm\beta) \le 4P$ and $\tilde w_q \ge P^{-1/2}$, so the first term on the right side of \eqref{eqp2.1} dominates the second. Thus, Lemma \ref{l4} and \eqref{eqp2.1} give
\[ \big| \mathcal F_N(q^{-1}\mathbf a + \bm\beta; \bm\xi) - \mathcal{G}_N(\bm\beta, q, \mathbf{a}; \bm\xi) \big| \lesssim_\eps \tilde w_q^{d-1}N^{2d-1+\eps}\Psi(\bm\beta)^{(3-d)/2}, \] 
uniformly in $\bm\xi$. Under the assumption $d \ge 9$, we deduce that
\begin{align*} 
& \int_{\mathfrak M_q} \sup_{\bm\xi} |\mathcal F_N(q^{-1}\mathbf a + \bm\beta; \bm\xi) - \mathcal{G}_N(\bm\beta, q, \mathbf{a}; \bm\xi)| \, \dbeta \\
&\lesssim_\eps \int_{\mathfrak M_q} \frac {\tilde w_q^{d-1} q^{(3-d)/2}N^{2d-1+\eps} }{(1 + N^2|\bm\beta|)^{(d-3)/2}} \,  \dbeta \lesssim_\eps \tilde w_q^{d-1}q^{(3-d)/2}N^{2d-7+\eps}.
\end{align*}
Using this bound and Lemma \ref{l10}, we obtain 
\begin{equation}\label{eqp2.3}
\Big\| \sup_{\lambda \in [\Lambda/2,\Lambda)} \big| \big( T_{\lambda}^{\mathbf{a}/q} - A_{\lambda}^{\mathbf{a}/q} \big) f \big| \Big\|_{2}
\lesssim_\eps \tilde w_q^{d-1}q^{(3-d)/2}N^{-1+\eps}.
\end{equation}

Note that we have also
\[ \mathcal{G}_N(\bm\beta, q, \mathbf{a}; \bm\xi) = \sum_{\mathbf m \in \mathbb Z^d} \mathbf 1_Q (q\bm\xi - \mathbf m) g(q; \mathbf {a, m}) V_N( \bm\beta; \bm\xi_{q,\mathbf m}), \]
where $\mathbf 1_Q$ is the indicator function of the unit cube $[-\frac 12, \frac 12)^d$. It is clear from this representation of $\mathcal G_N$ that its behavior changes abruptly as $\bm\xi$ moves around and $\mathbf m$ jumps from one lattice point to a neighboring one. To mitigate this effect, we now approximate $A_\lambda^{\mathbf{a}/q}$ by the convolution operator $B_\lambda^{\mathbf{a}/q}$ with Fourier multiplier
\[ \widehat{B_{\lambda}^{\mathbf{a}/q}}(\bm\xi) = \lambda^{3-d} e_q(-\lambda s(\mathbf a)) \int_{\mathfrak M_q} \mathcal{H}_N(\bm\beta, q, \mathbf{a}; \bm\xi) e(-\lambda s(\bm\beta)) \, \dbeta, \]
where
\[ \mathcal{H}_N(\bm\beta, q, \mathbf{a}; \bm\xi) = \sum_{\mathbf m \in \mathbb Z^d} \Phi (q\bm\xi - \mathbf m) g(q; \mathbf {a, m}) V_N( \bm\beta; \bm\xi_{q,\mathbf m}), \]
$\Phi$ being the smooth cutoff function that appears in the statements of Theorems \ref{thm2} and \ref{thm3}. The difference $\mathcal G_N - \mathcal H_N$ is supported on a set where $\frac 18 \le |q\xi_j - m_j| \le \frac 12$ for some $j$. For such $j$, Lemma \ref{l3} yields
\[ V_N(\bm\beta; \xi_j - m_j/q, 0) \lesssim_\eps q^{1/2}N^{3/2+\eps}. \]
We deduce that
\[  \sup_{\bm\xi} |\mathcal G_N( \bm\beta, q, \mathbf{a}; \bm\xi) - \mathcal H_N( \bm\beta, q, \mathbf{a}; \bm\xi)| \lesssim_\eps \frac {\tilde w_q^{d}q^{(1-d)/2}N^{2d-1/2+\eps}}{(1 + N^2|\bm\beta|)^{(d-1)/2}}, \]
and hence, 
\begin{align}\label{eqp2.4} 
& \int_{\mathfrak M_q} \sup_{\bm\xi} |\mathcal G_N( \bm\beta, q, \mathbf{a}; \bm\xi) - \mathcal H_N( \bm\beta, q, \mathbf{a}; \bm\xi| \, \dbeta \notag\\
&\lesssim_\eps \int_{\mathfrak M_q} \frac {\tilde w_q^{d} q^{(1-d)/2} N^{2d-1/2+\eps} }{(1 + N^2|\bm\beta|)^{(d-1)/2}} \, \dbeta \lesssim_\eps \tilde w_q^{d}q^{(1-d)/2}N^{2d-13/2+\eps}.
\end{align}
Lemma \ref{l10} and \eqref{eqp2.4} give
\begin{align}\label{eqp2.5}
\Big\| \sup_{\lambda \in [\Lambda/2,\Lambda)} \big| \big( A_{\lambda}^{\mathbf{a}/q} - B_{\lambda}^{\mathbf{a}/q} \big) f \big| \Big\|_{2}
\lesssim_\eps \tilde w_q^{d}q^{(1-d)/2}N^{-1/2+\eps}.
\end{align}

Next, we approximate $B_{\lambda}^{\mathbf{a}/q}$ by the convolution operator $M_{\lambda}^{\mathbf{a}/q}$ with multiplier  
\[ \widehat{M_{\lambda}^{\mathbf{a}/q}}(\bm\xi) = \lambda^{3-d}e_q(-\lambda s(\mathbf a))\sum_{\mathbf {m} \in \mathbb Z^d} \Phi( q\bm\xi - \mathbf m) g(q; \mathbf {a, m})J_{\lambda}(\bm\xi_{q,\mathbf m}; \mathbb R^3), \]
where
\[ J_{\lambda}(\bm\xi; {\mathfrak B}) = \int_{\mathfrak B} V_N(\bm\beta; \bm\xi) e(-\lambda s(\bm\beta)) \, \dbeta. \]
We can express $\widehat{B_{\lambda}^{\mathbf{a}/q}}(\bm\xi)$ in a matching form, with $J_{\lambda}(\bm\xi; \mathfrak M_q)$ in place of $J_{\lambda}(\bm\xi; \mathbb R^3)$. Thus, when $d \ge 7$, we deduce from Lemmas \ref{l1}, \ref{l3} and \ref{l10} that
\begin{align}\label{eqp2.6}
\Big\| \sup_{\lambda \in [\Lambda/2,\Lambda)} \big| \big( B_{\lambda}^{\mathbf{a}/q} - M_{\lambda}^{\mathbf{a}/q} \big) f \big| \Big\|_{2} 
&\lesssim_\eps \int_{\mathfrak M_q^c} \frac {\tilde w_q^{d}q^{-d/2}N^6 }{(1 + N^2|\bm\beta|)^{d/2-\eps}} \, \dbeta \notag\\
&\lesssim_\eps \tilde w_q^{d}q^{-3}P^{3-d/2+\eps}.
\end{align}
Here, $\mathfrak M_q^c$ denotes the complement of the box $\mathfrak M_q$ in $\mathbb R^3$. 

Using \eqref{eqp2.3}, \eqref{eqp2.5}, \eqref{eqp2.6}, and Lemma \ref{l2}, we conclude that
\begin{align*} 
&\sum_{\substack{1 \le \mathbf a \le q\\(q, a_1, a_2, a_3)=1}} \Big\| \sup_{\lambda \in [\Lambda/2,\Lambda)} \big| \big( T_{\lambda}^{\mathbf{a}/q} - M_{\lambda}^{\mathbf{a}/q} \big) f \big| \Big\|_{2} \\
&\lesssim_\eps \sum_{\substack{1 \le \mathbf a \le q\\(q, a_1, a_2, a_3)=1}} \big( q^{2-d}N^{-1+\eps} + q^{1-d}N^{-1/2+\eps} + q^{-(5+d)/2}P^{3-d/2+\eps} \big)w_q(\mathbf a)^{d-1} \\
&\lesssim_\eps q^{(5-d)/2}N^{-1/2+2\eps} + q^{-1}P^{3-d/2+2\eps}. 
\end{align*}
This completes the proof of the proposition when $d \ge 9$. 

\medskip

Suppose now that $d = 7$ or $8$. A quick examination of the above argument reveals that most of it carries without change. Indeed, the only place where a significant adjustment is needed is inequality \eqref{eqp2.3}, which changes to 
\begin{equation}\label{eqp2.7}
\Big\| \sup_{\lambda \in [\Lambda/2,\Lambda)} \big| \big( T_{\lambda}^{\mathbf{a}/q} - A_{\lambda}^{\mathbf{a}/q} \big) f \big| \Big\|_{2} \lesssim_\eps \tilde w_q^{d-1}q^{-3}P^{(9-d)/2}N^{-1+\eps}.
\end{equation}
Since the resulting contribution to the approximation error is still dominated by the contribution coming from inequality \eqref{eqp2.5}, this change does not affect the final result. \qed

\subsection{The main term}
\label{s5.2}

Recall Lemma \ref{l9}. Since $J_\lambda(\bm\xi; \mathbb R^3)$ in the definition of the operator $M_{\lambda}^{\mathbf{a}/q}$ is really the integral $I_N(\lambda; \bm\xi)$ in that lemma, we see that when $\lambda < \Lambda$, we can rewrite $\widehat{M_{\lambda}^{\mathbf{a}/q}}$ in a scale-independent form. Namely,
\[ \widehat{M_{\lambda}^{\mathbf{a}/q}}(\bm\xi) = c_d e_q(-\lambda s(\mathbf a))\sum_{\mathbf {m} \in \mathbb Z^d} \Phi( q\bm\xi - \mathbf m) g(q; \mathbf {a, m})\widetilde{\mathrm dS}\big( \lambda^{1/2}(\bm\xi - q^{-1}\mathbf{m}) \big), \]
where $c_d > 0$ and $\widetilde{\mathrm dS}(\bm\xi)$ are as in Lemma \ref{l9}. This representation allows us to give a quick proof of Proposition \ref{p3} and also verifies \eqref{eq3.9}. 

\begin{proof}[Proof of Proposition \ref{p3}]
The above form of the multiplier $\widehat{M_{\lambda}^{\mathbf{a}/q}}(\bm\xi)$ matches closely the form of the analogous multiplier in the work of Magyar, Stein and Wainger \cite{MSW}. In particular, the work in Section 3 of \cite{MSW} goes through for $M_{\lambda}^{\mathbf{a}/q}$ with minimal modifications. Using Lemma~\ref{l1} in place of the bound for the classical Gauss sum in the proof of \cite[Proposition 3.1(a)]{MSW}, we find that, for $p > d/(d-1)$, 
\[ \Big\| \sup_{\lambda \in \mathbb N} \big| M_{\lambda}^{\mathbf{a}/q} f \big| \Big\|_{p} \lesssim q^{-2d/p'}w_q(\mathbf a)^{2d/p'}\| f\|_p. \]
An appeal to Lemma \ref{l2} then completes the proof.
\end{proof}

\subsection{Proof of Proposition \ref{p4}}
\label{s5.3}

We revisit the dyadic maximal operator $T^*_{\mathfrak M(q; \mathbf{a})}$. By \eqref{eq3.2} and Minkowski's inequality, we have
\begin{equation}\label{eqp4.1} 
\| T^*_{\mathfrak M(q; \mathbf{a})}f \|_{p} \lesssim \Lambda^{3-d} \int_{\mathfrak M(q;\mathbf{a})} \left\| F_N(\bm\alpha; f) \right\|_{p} \, \dalpha.
\end{equation} 
Recall \eqref{eqp2.1} and the observation we made earlier that, when $q^{-1}\mathbf a + \bm\beta \in \mathfrak M(q;\mathbf{a})$, the second term on the right side of that inequality is superfluous. From \eqref{eq4.2} and \eqref{eqp2.1}, we get
\begin{equation}\label{eqp4.2}
\mathcal F_N(\bm\alpha; \bm\xi) \lesssim_\eps \tilde w_q^d N^{2d} \Psi(\bm\alpha)^{-d/2+\eps},
\end{equation} 
where $\tilde w_q$ is given by \eqref{eqp2.2} and $\Psi(\bm\alpha)$ is defined on $\mathfrak M(q; \mathbf{a})$ as 
\[ \Psi(\bm\alpha) = q + N^2|q\bm\alpha - \mathbf a|. \]

In $\ell^2(\mathbb Z^d)$, the Parseval--Plancherel identity and \eqref{eqp4.2} give
\begin{align*}
\left\| F_N(\bm\alpha; f) \right\|_{2}^2 &=
\int_{\mathbb T^d} |\mathcal F_N(\bm\alpha; \bm\xi) \hat f(\bm\xi) |^2 \, \mathrm d\bm\xi \notag\\
&\lesssim_\eps \tilde w_q^{2d} N^{4d}\Psi(\bm\alpha)^{-d+\eps} \int_{\mathbb T^d} |\hat f(\bm\xi) |^2 \, \mathrm d\bm\xi = \tilde w_q^{2d} N^{4d} \Psi(\bm\alpha)^{-d+\eps}\|f\|_{2}^2.
\end{align*}
We combine this inequality with the trivial $\ell^1$-bound
\[ \left\| F_N(\bm\alpha; f) \right\|_{1} \lesssim N^{2d}\|f\|_{1}. \]
When $1 < p < 2$, interpolation between these two inequalities yields
\begin{align}\label{eqp4.3}
\left\| F_N(\bm\alpha; f) \right\|_{p} \lesssim_\eps \tilde w_q^{2d/p'}N^{2d}\Psi(\bm\alpha)^{-d/p'+\eps}\|f\|_{p}, 
\end{align}
where $1/p' = 1 - 1/p$. 

Fix a function $f \in \ell^p(\mathbb Z^d)$ with $\| f\|_p = 1$. Applying \eqref{eqp4.3} to the right side of \eqref{eqp4.1}, we conclude that
\begin{align}\label{eqp4.4}
\| T^*_{\mathfrak M(q;\mathbf{a})}f \|_{p} &\lesssim_\eps \tilde w_q^{2d/p'}N^6 \int_{\mathfrak M(q; \mathbf{a})} \Psi(\bm\alpha)^{-d/p'+\eps} \, \dalpha \notag\\ 
&\lesssim_\eps q^{-d/p'+\eps}\tilde w_q^{2d/p'} \int_{|\bm\beta| \le P/q} (1 + |\bm\beta|)^{-d/p'+\eps} \, \dbeta \notag\\
&\lesssim_\eps q^{-2d/p'+\eps}w_q(\mathbf a)^{2d/p'},
\end{align} 
provided that $d/p' > 3$ and $\eps > 0$ is chosen sufficiently small. Finally, we sum \eqref{eqp4.4} over all major arcs to bound $\| T^*_{\mathfrak M}f \|_{p}$. When $d/p' > 3$, we obtain
\begin{align*}
\| T^*_{\mathfrak M}f \|_{p} &\le \sum_{q, \mathbf{a}} \| T^*_{\mathfrak M(q; \mathbf{a})}f \|_{p} \lesssim \sum_{\substack{1 \le \mathbf a \le q \le P \\(q, a_1, a_2, a_3)=1}} q^{-2d/p'+\eps} w_q(\mathbf a)^{2d/p'} \\  
&\lesssim_\eps \sum_{q \le P} q^{-d/p'+2+\eps} \lesssim_\eps 1,
\end{align*} 
after using Lemma \ref{l2} once again. Since the condition $d/p' > 3$ is equivalent to the hypothesis $p > \frac d{d-3}$, the proposition follows. \qed

\section{Counting lattice points: Proof of Theorem \ref{thm2}}
\label{s6}

Similarly to \eqref{eq3.1}, we have
\[ \widehat{T_\lambda}(\bm\xi, \bm\eta) = R_\lambda(\bm\xi, \bm\eta; \mathbb T^3), \]
where
\[ R_\lambda(\bm\xi, \bm\eta; \mathfrak B) = \lambda^{3-d} \int_{\mathfrak B} \mathcal F_N(\bm\alpha; \bm\xi, \bm\eta) e(-\lambda s(\bm\alpha)) \, \dalpha, \]
with $N = \lambda^{1/2}$ and
\[ \mathcal F_N(\bm\alpha; \bm\xi, \bm\eta) = \prod_{j=1}^{d} S_N(\bm\alpha; \xi_j, \eta_j).\]
We apply the circle method to $R_\lambda(\bm\xi, \bm\eta; \mathbb T^3)$, using a Hardy--Littlewood decomposition given by \eqref{eq2.8m} with $P = N^{2/7}$. Note that with this choice, Lemma \ref{l6} yields
\begin{equation}\label{eq6.0}
\sup_{\bm\alpha \in \mathfrak m} |S_N(\bm\alpha; \xi, \eta)| \lesssim_\eps N^{13/7 + \eps}. 
\end{equation}

It is straightforward to adapt the proof of Proposition \ref{p2} in Section \ref{s5.1} to show that
\[ R_\lambda(\bm\xi, \bm\eta; \mathfrak M) = \lambda^{3-d} \sum_{q \le P} \sum_{\mathbf{m,n} \in \mathbb Z^d} G_\lambda(q; \mathbf{m,n})\Phi_q(\bm\xi_{q,\mathbf m})\Phi_q(\bm\eta_{q,\mathbf n}) J_\lambda(\bm\xi_{q,\mathbf m}, \bm\eta_{q,\mathbf n}) + O_\eps\big(P^{-1/2 + \eps}\big), \]
where
\[ J_{\lambda}(\bm\xi, \bm\eta) = \int_{\mathbb R^3} V_N(\bm\beta; \bm\xi, \bm\eta) e(-\lambda s(\bm\beta)) \, \dbeta. \]
We have
\[ J_{\lambda}(\bm\xi, \bm\eta) = \lambda^{d-3} I_{\lambda}(\bm\xi, \bm\eta), \]
where $I_\lambda(\bm\xi, \bm\eta)$ is the integral appearing in the statement of Theorem \ref{thm2}. Since Lemmas \ref{l1}--\ref{l3} give
\begin{equation}\label{eq6.1} 
G_\lambda(q; \mathbf{m,n}) \lesssim_\eps q^{-d/2 + 2 + \eps}, \quad I_{\lambda}(\bm\xi, \bm\eta) \lesssim 1, 
\end{equation}
we conclude that
\[ R_\lambda(\bm\xi, \bm\eta; \mathfrak M) = \sum_{q = 1}^{\infty} \sum_{\mathbf{m,n} \in \mathbb Z^d} G_\lambda(q; \mathbf{m,n})\Phi_q(\bm\xi_{q,\mathbf m})\Phi_q(\bm\eta_{q,\mathbf n}) I_\lambda(\bm\xi_{q,\mathbf m}, \bm\eta_{q,\mathbf n}) + O_\eps\big(P^{-1/2 + \eps}\big). \]

On the other hand, by \eqref{eq4.2}, \eqref{eq6.0} and a variant of \eqref{eqp1.1}, 
\[ R_\lambda(\bm\xi, \bm\eta; \mathfrak m) \lesssim_\eps N^{-6 + (6-d)/7+\eps} \int_{\mathbb T^3} |S_N(\bm\alpha; \xi_j, \eta_j)|^6 \, \dalpha \]
for some $j \le d$. Lemma \ref{l8} then gives
\[ R_\lambda(\bm\xi, \bm\eta; \mathfrak m) \lesssim_\eps N^{(6-d)/7+\eps}, \]
and this completes the proof of \eqref{eqi.3}.

When $\bm\xi = \bm\eta = \bm0$, the sum over $\mathbf{m,n}$ on the right side of \eqref{eqi.3} always picks its contribution from the term $\mathbf m = \mathbf n = \bm 0$. Thus, 
\begin{equation}\label{eq6.2} 
\widehat{T_\lambda}(\bm 0, \bm 0) = \mathfrak S(\lambda) I_\lambda(\bm 0, \bm 0) + O_\eps \big(\lambda^{-1/14+\eps} \big), 
\end{equation}
where 
\[ \mathfrak S(\lambda) = \sum_{q=1}^{\infty} G_\lambda(q; \bm0, \bm0). \]
Since $I_\lambda(\bm 0, \bm 0)$ is also the integral $I_1(\lambda; \bm 0)$ in the notation of Lemma \ref{l9}, that lemma gives
\[ 1 \lesssim I_\lambda(\bm 0, \bm 0) \lesssim 1. \]
The second claim of Theorem~\ref{thm2} is therefore an immediate consequence of \eqref{eq6.2} and the following result.

\begin{lemma}
Let $d \ge 7$ and $\lambda \in \mathbb N$ be even. The singular series $\mathfrak S(\lambda)$ is absolutely convergent and satisfies
\begin{equation}\label{eq6.3} 
1 \lesssim \mathfrak S(\lambda) \lesssim 1.  
\end{equation}
\end{lemma}

\begin{proof}[Sketch of proof]
The absolute convergence of $\mathfrak S(\lambda)$ and the upper bound in \eqref{eq6.3} follow from~\eqref{eq6.1}. As to the lower bound, we observe that similarly to Lemmas 2.10 and 2.11 in Vaughan \cite{Vaug97}, one can show that $G_\lambda(q) := G_\lambda(q; \bm0, \bm0)$ is multiplicative in $q$. Together with the absolute convergence of the series, this allows us to factor $\mathfrak S(\lambda)$ as an Euler product:
\[ \mathfrak S(\lambda) = \sum_{q=1}^{\infty} G_\lambda(q) = \prod_p \big( 1 + G_\lambda(p) + G_\lambda(p^2) + \cdots  \big) =: \prod_p T(p). \]
Similarly to the proofs of Theorem 2.4 and Lemma 2.12 in Vaughan \cite{Vaug97}, we then see that $T(p) \ge 0$ and
\[ T(p) =  \lim_{t \to \infty} p^{(3-2d)t}\nu_d(p^t; \lambda), \] 
where $\nu_d(q; \lambda)$ is the number of solutions $\mathbf{x,y} \in \mathbb Z_q^d$ of the simultaneous congruences
\[ \sum_{i=1}^d x_i^2 \equiv \sum_{i=1}^d y_i^2 \equiv 2\sum_{i=1}^d x_iy_i \equiv \lambda \pmod q. \] 
Therefore, it remains to show that, for $d \ge 7$, $t \ge 3$, and $\lambda$ even, we have
\[ \nu_d(p^t; \lambda) \ge \begin{cases}
p^{(t-1)(2d-3)} &\text{if } p > 2;\\
8 \cdot 2^{(t-2)(2d-3)} &\text{if } p=2. 
\end{cases} \]
The proof of this inequality is a standard Hensel-type argument that first constructs a solution modulo $p$ (resp., modulo $8$) and then lifts that solution to $p^{(t-1)(2d-3)}$ solutions modulo $p^t$ (resp., $2^{(t-2)(2d-3) + 3}$ solutions modulo $2^t$). We omit the details and refer the reader to Lemmas 2.13--2.15 in Vaughan \cite{Vaug97} and Lemmas 5--7 in Raghavan \cite{Ragh59}. In particular, Lemma 2.15 in \cite{Vaug97} is used to construct the initial solution modulo $p$ for odd $p$, while the proofs in \cite{Ragh59} are indicative of the lifting argument (though considerably more technical due to the more general setting in that paper). 
\end{proof}

\section{Counting equilateral triangles}
\label{s7}

Motivated by the sharpness example for the Erd\H{o}s distance problem, where distances are counted in an integer lattice, it is natural to count other point configurations in such a lattice. In fact, counting triangles in the integer lattice has allowed for the only non-trivial sharpness examples for Falconer type theorems for triangles \cite{GILP15}. One has to be careful with counting equilateral triangles in $\mathbb{Z}^d$, for there are none in $\mathbb{Z}^2$; however, they do exist in higher dimensions \cite{Bee92}. Characterizations have been given for equilateral triangles in $\mathbb{Z}^3$ \cite{Ion07,CI08} and in $\mathbb{Z}^4$ \cite{Ion15}. Moreover, in dimensions $d=3$ and $4$ and for small values of $n$, all integer equilateral triangles in the cube $[0,n]^d$ have been counted using the Ehrhart polynomial: see \cite{CI08,Ion08,Ion15}. In those papers, the authors make also conjectures for the growth of the total number of such triangles as $n \to \infty$. Using the count of equilateral triangles established in Theorem \ref{thm2}, we can answer such questions for $d\geq 7$ and obtain an asymptotic upper bound of $n^{3d-4}$. This is achieved by observing that equilateral triangles pinned at every point in $[0,n]^d \cap \mathbb{Z}^d$ must have a side length squared $\lambda^2\in\lbrace 1, 2, \ldots, dn^2 \rbrace$ and each such triangle appears no more often than $\lambda^{2(d-3)}$ times. If our upper bound were to hold all the way down to $d=3$, we would obtain an asymptotic upper bound of $n^5$, which would match the conjecture made in \cite{Ion08}.

Falconer type theorems for triangles, established through incidence estimates, allow for counting triangles in homogeneous and well distributed sets through a certain continuous to discrete transference mechanism \cite{HI05,IL05}. Dense subsets of the integer lattice, such as $[0,n]^d \cap \mathbb{Z}^d$, are stereotypical homogeneous and well distributed sets. Through the best incidence estimates \cite{GI12,GGIP15} a fixed equilateral triangle in $[0,n]^d \cap \mathbb{Z}^d$ appears asymptotically no more than $n^{3d - \frac{12d}{3d+1}}$ times when $d\geq 2$, while here, through Theorem \ref{thm2}, we get an asymptotic upper bound of $n^{3d-6}$ when $d\geq 7$, which is always smaller in corresponding dimensions. If the incidence bounds, and therefore the Falconer type theorem for triangles, held true for sets of Hausdorff dimension down to the threshold $\frac{d}{2}$, as is conjectured in the case of the distance, then the transference mechanisms would yield an upper bound coinciding with what we obtain in this paper. However, as shown in \cite{GILP15}, in the plane there is a sharpness threshold of $\frac{3}{2}$ as opposed to $1$, for the Falconer type theorem for triangles. Is there yet again different behavior in lower dimensions, or are equilateral triangles perhaps not the extremal cases for the incidence theorems?

\section{Final remarks}
\label{s8}

We now return to the question of relaxing the restrictions on $p$ and $d$ in Theorems \ref{thm1} and \ref{thm3}. A look back at Propositions \ref{p1}--\ref{p4} shows that the constraint $d \ge 9$ is imposed by the treatment of the minor arcs in Section 4, where it is made necessary by the use of Lemmas \ref{l10} and \ref{l11}. Were the supremum over $\xi$ not present on the right side of \eqref{eqp1.1}, we would have been able to refer to Lemma \ref{l8} instead of Lemma \ref{l7} to obtain versions of Proposition \ref{p1} and Theorems~\ref{thm1} and~\ref{thm3} for $d \ge 7$ and $p > \max\big( \frac {36}{d+12}, \frac {d+6}{d} \big)$.

One way to circumvent the above issue is to switch from a conventional application of the circle method to one where all the arcs are treated as major. This idea goes back to the work of Kloosterman \cite{Kloo26} on representations of the integers by diagonal forms in four variables; in that context, it is known as the {\em Kloosterman refinement} of the circle method. Here, we will use a very basic form of this idea to demonstrate how one can leverage a hypothetical strong version of inequality \eqref{eqp2.1} above to bound our maximal operator for $d \ge 7$ and $p > d/(d-3)$. 

\medskip

We retain all the notation introduced in Sections \ref{s3}--\ref{s5}, and in particular, the definitions of $\mathfrak M$ and $\mathfrak m$ in \eqref{eq2.8m}, though here we choose $P = N^{\theta}$, where $\theta < 1$ will be fixed shortly. Also, we write $\mathfrak L(\mathbf{q;a})$ for the major arc $\mathfrak N(N, N; \mathbf{q,a})$ corresponding to the choice $P = N$ and define
\[ \mathfrak L = \bigcup_{[q_1,q_2,q_3] \le P} \bigcup_{\substack{1 \le \mathbf a \le \mathbf q\\(a_i,q_i)=1}} \mathfrak L(\mathbf{q;a}). \] 
We now assume the following stronger version of inequality \eqref{eqp2.1}: 
If $1 \le \mathbf a \le \mathbf q \le N$, with $(a_i,q_i)=1$, and  $\bm\alpha \in \mathfrak L(\mathbf{q;a})$, then for all $\xi,\eta$,
\begin{equation}\label{eq8.1}
S_N(\bm\alpha; \xi, \eta) \lesssim \tilde w_q N^{2+\eps}(q + N^2|q\bm\alpha - \mathbf b|)^{-1/2} + N^{1+\eps},    
\end{equation}  
where 
\[ q = [q_1, q_2, q_3], \quad b_i = \frac {a_iq}{q_i} \quad (1 \le i \le 3). \]    
We will use this hypothesis to obtain an alternative version of Proposition \ref{p1} using an argument similar to that we used in Section \ref{s5.3} to establish Proposition \ref{p4}. 

Let $f \in \ell^p(\mathbb Z^d)$, $d/(d-3) < p \le 2$, with $\|f\|_p = 1$. Under the hypothesis \eqref{eq8.1}, we find similarly to \eqref{eqp4.3} that 
\begin{equation}\label{eq8.2} 
 \| F_N(\bm\alpha; f) \|_p \lesssim_\eps N^{2d+\eps}\big( q^{-2d/p'}w_q(\mathbf b)^{2d/p'}(1 + N^2|\bm\beta|)^{-d/p'} + N^{-2d/p'} \big),  
\end{equation} 
where $1/p' = 1 - 1/p$ and $\beta_i = \alpha_i - a_i/q_i = \alpha_i - b_i/q$. When $d/p' > 3$ (recall that this inequality is equivalent to $p > d/(d-3)$), we deduce that
\begin{equation}\label{eq8.3} 
\int_{\mathfrak L(\mathbf q; \mathbf a)} \| F_N(\bm\alpha; f) \|_p \, \dalpha \lesssim_\eps  N^{2d - 6 + \eps} \big( q^{-2d/p'}w_q(\mathbf b)^{2d/p'} + (q_1q_2q_3)^{-1}N^{3 -2d/p'} \big). 
\end{equation} 
By Dirichlet's theorem on Diophantine approximation, the arcs $\mathfrak L(\mathbf q; \mathbf a)$ with $1 \le \mathbf a \le \mathbf q \le N$ cover $\mathbb T^3$.
Hence, summing \eqref{eq8.3} over all choices of $\mathbf{q,a}$ with $q = [q_1,q_2,q_3] > P$, we get
\begin{align}\label{eq8.4} 
\int_{\mathfrak m \setminus \mathfrak L} \| F_N(\bm\alpha; f) \|_p \, \dalpha &\lesssim_\eps N^{2d-6+\eps} \bigg\{ \sum_{q > P} q^{-2d/p'}\sum_{\substack{1 \le \mathbf a \le q\\(q, a_1 , a_2, a_3) = 1}}w_q(\mathbf a)^{2d/p'}  + \sum_{1 \le \mathbf q \le N} N^{3 - 2d/p} \bigg\} \notag\\
&\lesssim_\eps N^{2d-6+\eps} \bigg\{ \sum_{q > P} q^{2-d/p'+\eps}  + N^{6 - 2d/p'} \bigg\}  \lesssim_\eps N^{2d-6+\eps}P^{3-d/p'}, 
\end{align} 
by an appeal to Lemma \ref{l2}. On the other hand, when $[q_1,q_2,q_3] \le P$, \eqref{eq8.2} gives
\begin{align*} 
&\int_{\mathfrak L(\mathbf q; \mathbf a) \setminus \mathfrak M(q; \mathbf b)} \| F_N(\bm\alpha; f) \|_p \, \dalpha \\
&\lesssim_\eps  N^{2d - 6 + \eps} \big( q^{-3-d/p'}w_q(\mathbf b)^{2d/p'}P^{3-d/p'} + (q_1q_2q_3)^{-1}N^{3 -2d/p'} \big),
\end{align*}
whence
\begin{align}\label{eq8.5} 
\int_{\mathfrak m \cap \mathfrak L} \| F_N(\bm\alpha; f) \|_p \, \dalpha \lesssim_\eps N^{2d-6+\eps} P^{3-d/p'}. 
\end{align} 

As in the proof of Proposition \ref{p4}, combining \eqref{eq8.4} and \eqref{eq8.5}, we obtain a version of Proposition \ref{p1} for 
\[ d \ge 7, \quad p > \frac d{d-3}, \quad \alpha_p = \frac {\theta}2 \left( \frac d{p'} - 3 \right). \] 
The proof of Proposition \ref{p3} is independent of the choice of $P$, and the above argument shows that Proposition \ref{p4} remains true for all $P \le N$. As to Proposition \ref{p2}, it is easy to check that its proof also works for any $P \le N$, though the value of $\beta_2(d)$ is impacted by the value of $\theta$. When $P = N^{\theta}$, the argument in Section \ref{s5.1} yields
\[ \beta_2(d) = \min \bigg( \frac {2 + \theta(d-9)}4, \frac 14 , \frac {\theta(d-6)}4 \bigg). \]
Therefore, when $d \ge 8$, any choice of $\theta \in [\frac 12, 1)$ will result in $\beta_2(d) = \frac 14$, while for $d = 7$, the optimal choice of $\theta$ is $\theta=\frac 23$, resulting in $\beta_2(7) = \frac 16$. We summarize these observations in the following proposition.

\begin{prop}\label{pC}
Assume that inequality \eqref{eq8.1} above holds for all $\xi,\eta \in \mathbb R$ and all $\bm\alpha \in \mathfrak L(\mathbf q; \mathbf a)$ with $1 \le \mathbf a \le \mathbf q \le N$, $(a_i,q_i) = 1$. Then the conclusions of Theorems \ref{thm1} and \ref{thm3} hold for $p \ge 7$ and $p > \frac d{d-3}$. Moreover, the value of $\delta_2$ in \eqref{eqi.6} can be chosen as $\delta_2 = \min\big( \frac 14, \frac 16(d-6) \big)$.
\end{prop}

We remark that while our hypothetical bound \eqref{eq8.1} is quite strong and is not even close to what is presently known about $S_N(\bm\alpha; \xi, \eta)$, it represents a reasonable conjecture. Indeed, a strong form of the analogous bound for the one-dimensional Weyl sum
\[ \sum_{|x| \le N} e(\alpha x^2 + \xi x) \]
is known from the work of Vaughan \cite{Vaug09}.

\bibliographystyle{amsplain}

\end{document}